\documentclass[preprint,12pt]{elsarticle}

\usepackage{amssymb}
\usepackage{amsmath}
\usepackage{mathtools}
\DeclareMathAlphabet{\mathpzc}{OT1}{pzc}{m}{it}
\usepackage{calc} %
\usepackage{amsthm}

\usepackage{siunitx}
\usepackage[hidelinks]{hyperref}
\usepackage{cleveref}
\usepackage{xcolor}
\usepackage{xargs}
\usepackage{soul}
\usepackage{subcaption}
\usepackage{macros}
\usepackage{nicefrac}
\usepackage{csquotes}

\usepackage[inline]{enumitem}
\usepackage[mode=buildnew,subpreambles=true]{standalone}
\usepackage{shellesc}
\usepackage{tikz,tikzscale}
\usetikzlibrary{arrows,calc,fit,patterns,positioning,spy,backgrounds,decorations.fractals,shapes.misc}
\usetikzlibrary{shapes.geometric}
\tikzset{boximg/.style={remember picture,black,thick,draw,inner sep=0pt,outer sep=0pt}}
\usepackage{pgf}
\usepackage{pgfplots}
\usepgfplotslibrary{statistics}
\usepgfplotslibrary{fillbetween}
\pgfplotsset{lua backend=true}

\pgfplotsset{select coords between index/.style 2 args={
    x filter/.code={
      \ifnum\coordindex<#1\fi
      \ifnum\coordindex>#2\fi
    }
}}

\usetikzlibrary{colorbrewer}
\usepgfplotslibrary{colorbrewer}
\usepgfplotslibrary{colormaps}
\usepgfplotslibrary{groupplots,dateplot}
\pgfplotsset{
	compat=newest,
  cycle list/Dark2,
}

\pgfplotsset{
  	/pgfplots/colormap={redtoblue}{
			rgb=(
			0.12548999999999999,
			0.26274500000000001,
			0.50196099999999999)
			rgb=(
			0.57254899999999997,
			0.819608,
			0.87843099999999996)
			rgb=(
			0.94117600000000001,
			0.98431400000000002,
			0.98823499999999997)
			rgb=(
			0.76078400000000002,
			0.29411799999999999,
			0.21176500000000001)
			rgb=(
			0.54901999999999995,
			0.066667000000000004,
			0.098039000000000001)
    }
}

\pgfplotsset{
	/pgfplots/colormap={viridis}{
		rgb=(0.267004, 0.004874, 0.329415)
		rgb=(0.268510, 0.009605, 0.335427)
		rgb=(0.269944, 0.014625, 0.341379)
		rgb=(0.271305, 0.019942, 0.347269)
		rgb=(0.272594, 0.025563, 0.353093)
		rgb=(0.273809, 0.031497, 0.358853)
		rgb=(0.274952, 0.037752, 0.364543)
		rgb=(0.276022, 0.044167, 0.370164)
		rgb=(0.277018, 0.050344, 0.375715)
		rgb=(0.277941, 0.056324, 0.381191)
		rgb=(0.278791, 0.062145, 0.386592)
		rgb=(0.279566, 0.067836, 0.391917)
		rgb=(0.280267, 0.073417, 0.397163)
		rgb=(0.280894, 0.078907, 0.402329)
		rgb=(0.281446, 0.084320, 0.407414)
		rgb=(0.281924, 0.089666, 0.412415)
		rgb=(0.282327, 0.094955, 0.417331)
		rgb=(0.282656, 0.100196, 0.422160)
		rgb=(0.282910, 0.105393, 0.426902)
		rgb=(0.283091, 0.110553, 0.431554)
		rgb=(0.283197, 0.115680, 0.436115)
		rgb=(0.283229, 0.120777, 0.440584)
		rgb=(0.283187, 0.125848, 0.444960)
		rgb=(0.283072, 0.130895, 0.449241)
		rgb=(0.282884, 0.135920, 0.453427)
		rgb=(0.282623, 0.140926, 0.457517)
		rgb=(0.282290, 0.145912, 0.461510)
		rgb=(0.281887, 0.150881, 0.465405)
		rgb=(0.281412, 0.155834, 0.469201)
		rgb=(0.280868, 0.160771, 0.472899)
		rgb=(0.280255, 0.165693, 0.476498)
		rgb=(0.279574, 0.170599, 0.479997)
		rgb=(0.278826, 0.175490, 0.483397)
		rgb=(0.278012, 0.180367, 0.486697)
		rgb=(0.277134, 0.185228, 0.489898)
		rgb=(0.276194, 0.190074, 0.493001)
		rgb=(0.275191, 0.194905, 0.496005)
		rgb=(0.274128, 0.199721, 0.498911)
		rgb=(0.273006, 0.204520, 0.501721)
		rgb=(0.271828, 0.209303, 0.504434)
		rgb=(0.270595, 0.214069, 0.507052)
		rgb=(0.269308, 0.218818, 0.509577)
		rgb=(0.267968, 0.223549, 0.512008)
		rgb=(0.266580, 0.228262, 0.514349)
		rgb=(0.265145, 0.232956, 0.516599)
		rgb=(0.263663, 0.237631, 0.518762)
		rgb=(0.262138, 0.242286, 0.520837)
		rgb=(0.260571, 0.246922, 0.522828)
		rgb=(0.258965, 0.251537, 0.524736)
		rgb=(0.257322, 0.256130, 0.526563)
		rgb=(0.255645, 0.260703, 0.528312)
		rgb=(0.253935, 0.265254, 0.529983)
		rgb=(0.252194, 0.269783, 0.531579)
		rgb=(0.250425, 0.274290, 0.533103)
		rgb=(0.248629, 0.278775, 0.534556)
		rgb=(0.246811, 0.283237, 0.535941)
		rgb=(0.244972, 0.287675, 0.537260)
		rgb=(0.243113, 0.292092, 0.538516)
		rgb=(0.241237, 0.296485, 0.539709)
		rgb=(0.239346, 0.300855, 0.540844)
		rgb=(0.237441, 0.305202, 0.541921)
		rgb=(0.235526, 0.309527, 0.542944)
		rgb=(0.233603, 0.313828, 0.543914)
		rgb=(0.231674, 0.318106, 0.544834)
		rgb=(0.229739, 0.322361, 0.545706)
		rgb=(0.227802, 0.326594, 0.546532)
		rgb=(0.225863, 0.330805, 0.547314)
		rgb=(0.223925, 0.334994, 0.548053)
		rgb=(0.221989, 0.339161, 0.548752)
		rgb=(0.220057, 0.343307, 0.549413)
		rgb=(0.218130, 0.347432, 0.550038)
		rgb=(0.216210, 0.351535, 0.550627)
		rgb=(0.214298, 0.355619, 0.551184)
		rgb=(0.212395, 0.359683, 0.551710)
		rgb=(0.210503, 0.363727, 0.552206)
		rgb=(0.208623, 0.367752, 0.552675)
		rgb=(0.206756, 0.371758, 0.553117)
		rgb=(0.204903, 0.375746, 0.553533)
		rgb=(0.203063, 0.379716, 0.553925)
		rgb=(0.201239, 0.383670, 0.554294)
		rgb=(0.199430, 0.387607, 0.554642)
		rgb=(0.197636, 0.391528, 0.554969)
		rgb=(0.195860, 0.395433, 0.555276)
		rgb=(0.194100, 0.399323, 0.555565)
		rgb=(0.192357, 0.403199, 0.555836)
		rgb=(0.190631, 0.407061, 0.556089)
		rgb=(0.188923, 0.410910, 0.556326)
		rgb=(0.187231, 0.414746, 0.556547)
		rgb=(0.185556, 0.418570, 0.556753)
		rgb=(0.183898, 0.422383, 0.556944)
		rgb=(0.182256, 0.426184, 0.557120)
		rgb=(0.180629, 0.429975, 0.557282)
		rgb=(0.179019, 0.433756, 0.557430)
		rgb=(0.177423, 0.437527, 0.557565)
		rgb=(0.175841, 0.441290, 0.557685)
		rgb=(0.174274, 0.445044, 0.557792)
		rgb=(0.172719, 0.448791, 0.557885)
		rgb=(0.171176, 0.452530, 0.557965)
		rgb=(0.169646, 0.456262, 0.558030)
		rgb=(0.168126, 0.459988, 0.558082)
		rgb=(0.166617, 0.463708, 0.558119)
		rgb=(0.165117, 0.467423, 0.558141)
		rgb=(0.163625, 0.471133, 0.558148)
		rgb=(0.162142, 0.474838, 0.558140)
		rgb=(0.160665, 0.478540, 0.558115)
		rgb=(0.159194, 0.482237, 0.558073)
		rgb=(0.157729, 0.485932, 0.558013)
		rgb=(0.156270, 0.489624, 0.557936)
		rgb=(0.154815, 0.493313, 0.557840)
		rgb=(0.153364, 0.497000, 0.557724)
		rgb=(0.151918, 0.500685, 0.557587)
		rgb=(0.150476, 0.504369, 0.557430)
		rgb=(0.149039, 0.508051, 0.557250)
		rgb=(0.147607, 0.511733, 0.557049)
		rgb=(0.146180, 0.515413, 0.556823)
		rgb=(0.144759, 0.519093, 0.556572)
		rgb=(0.143343, 0.522773, 0.556295)
		rgb=(0.141935, 0.526453, 0.555991)
		rgb=(0.140536, 0.530132, 0.555659)
		rgb=(0.139147, 0.533812, 0.555298)
		rgb=(0.137770, 0.537492, 0.554906)
		rgb=(0.136408, 0.541173, 0.554483)
		rgb=(0.135066, 0.544853, 0.554029)
		rgb=(0.133743, 0.548535, 0.553541)
		rgb=(0.132444, 0.552216, 0.553018)
		rgb=(0.131172, 0.555899, 0.552459)
		rgb=(0.129933, 0.559582, 0.551864)
		rgb=(0.128729, 0.563265, 0.551229)
		rgb=(0.127568, 0.566949, 0.550556)
		rgb=(0.126453, 0.570633, 0.549841)
		rgb=(0.125394, 0.574318, 0.549086)
		rgb=(0.124395, 0.578002, 0.548287)
		rgb=(0.123463, 0.581687, 0.547445)
		rgb=(0.122606, 0.585371, 0.546557)
		rgb=(0.121831, 0.589055, 0.545623)
		rgb=(0.121148, 0.592739, 0.544641)
		rgb=(0.120565, 0.596422, 0.543611)
		rgb=(0.120092, 0.600104, 0.542530)
		rgb=(0.119738, 0.603785, 0.541400)
		rgb=(0.119512, 0.607464, 0.540218)
		rgb=(0.119423, 0.611141, 0.538982)
		rgb=(0.119483, 0.614817, 0.537692)
		rgb=(0.119699, 0.618490, 0.536347)
		rgb=(0.120081, 0.622161, 0.534946)
		rgb=(0.120638, 0.625828, 0.533488)
		rgb=(0.121380, 0.629492, 0.531973)
		rgb=(0.122312, 0.633153, 0.530398)
		rgb=(0.123444, 0.636809, 0.528763)
		rgb=(0.124780, 0.640461, 0.527068)
		rgb=(0.126326, 0.644107, 0.525311)
		rgb=(0.128087, 0.647749, 0.523491)
		rgb=(0.130067, 0.651384, 0.521608)
		rgb=(0.132268, 0.655014, 0.519661)
		rgb=(0.134692, 0.658636, 0.517649)
		rgb=(0.137339, 0.662252, 0.515571)
		rgb=(0.140210, 0.665859, 0.513427)
		rgb=(0.143303, 0.669459, 0.511215)
		rgb=(0.146616, 0.673050, 0.508936)
		rgb=(0.150148, 0.676631, 0.506589)
		rgb=(0.153894, 0.680203, 0.504172)
		rgb=(0.157851, 0.683765, 0.501686)
		rgb=(0.162016, 0.687316, 0.499129)
		rgb=(0.166383, 0.690856, 0.496502)
		rgb=(0.170948, 0.694384, 0.493803)
		rgb=(0.175707, 0.697900, 0.491033)
		rgb=(0.180653, 0.701402, 0.488189)
		rgb=(0.185783, 0.704891, 0.485273)
		rgb=(0.191090, 0.708366, 0.482284)
		rgb=(0.196571, 0.711827, 0.479221)
		rgb=(0.202219, 0.715272, 0.476084)
		rgb=(0.208030, 0.718701, 0.472873)
		rgb=(0.214000, 0.722114, 0.469588)
		rgb=(0.220124, 0.725509, 0.466226)
		rgb=(0.226397, 0.728888, 0.462789)
		rgb=(0.232815, 0.732247, 0.459277)
		rgb=(0.239374, 0.735588, 0.455688)
		rgb=(0.246070, 0.738910, 0.452024)
		rgb=(0.252899, 0.742211, 0.448284)
		rgb=(0.259857, 0.745492, 0.444467)
		rgb=(0.266941, 0.748751, 0.440573)
		rgb=(0.274149, 0.751988, 0.436601)
		rgb=(0.281477, 0.755203, 0.432552)
		rgb=(0.288921, 0.758394, 0.428426)
		rgb=(0.296479, 0.761561, 0.424223)
		rgb=(0.304148, 0.764704, 0.419943)
		rgb=(0.311925, 0.767822, 0.415586)
		rgb=(0.319809, 0.770914, 0.411152)
		rgb=(0.327796, 0.773980, 0.406640)
		rgb=(0.335885, 0.777018, 0.402049)
		rgb=(0.344074, 0.780029, 0.397381)
		rgb=(0.352360, 0.783011, 0.392636)
		rgb=(0.360741, 0.785964, 0.387814)
		rgb=(0.369214, 0.788888, 0.382914)
		rgb=(0.377779, 0.791781, 0.377939)
		rgb=(0.386433, 0.794644, 0.372886)
		rgb=(0.395174, 0.797475, 0.367757)
		rgb=(0.404001, 0.800275, 0.362552)
		rgb=(0.412913, 0.803041, 0.357269)
		rgb=(0.421908, 0.805774, 0.351910)
		rgb=(0.430983, 0.808473, 0.346476)
		rgb=(0.440137, 0.811138, 0.340967)
		rgb=(0.449368, 0.813768, 0.335384)
		rgb=(0.458674, 0.816363, 0.329727)
		rgb=(0.468053, 0.818921, 0.323998)
		rgb=(0.477504, 0.821444, 0.318195)
		rgb=(0.487026, 0.823929, 0.312321)
		rgb=(0.496615, 0.826376, 0.306377)
		rgb=(0.506271, 0.828786, 0.300362)
		rgb=(0.515992, 0.831158, 0.294279)
		rgb=(0.525776, 0.833491, 0.288127)
		rgb=(0.535621, 0.835785, 0.281908)
		rgb=(0.545524, 0.838039, 0.275626)
		rgb=(0.555484, 0.840254, 0.269281)
		rgb=(0.565498, 0.842430, 0.262877)
		rgb=(0.575563, 0.844566, 0.256415)
		rgb=(0.585678, 0.846661, 0.249897)
		rgb=(0.595839, 0.848717, 0.243329)
		rgb=(0.606045, 0.850733, 0.236712)
		rgb=(0.616293, 0.852709, 0.230052)
		rgb=(0.626579, 0.854645, 0.223353)
		rgb=(0.636902, 0.856542, 0.216620)
		rgb=(0.647257, 0.858400, 0.209861)
		rgb=(0.657642, 0.860219, 0.203082)
		rgb=(0.668054, 0.861999, 0.196293)
		rgb=(0.678489, 0.863742, 0.189503)
		rgb=(0.688944, 0.865448, 0.182725)
		rgb=(0.699415, 0.867117, 0.175971)
		rgb=(0.709898, 0.868751, 0.169257)
		rgb=(0.720391, 0.870350, 0.162603)
		rgb=(0.730889, 0.871916, 0.156029)
		rgb=(0.741388, 0.873449, 0.149561)
		rgb=(0.751884, 0.874951, 0.143228)
		rgb=(0.762373, 0.876424, 0.137064)
		rgb=(0.772852, 0.877868, 0.131109)
		rgb=(0.783315, 0.879285, 0.125405)
		rgb=(0.793760, 0.880678, 0.120005)
		rgb=(0.804182, 0.882046, 0.114965)
		rgb=(0.814576, 0.883393, 0.110347)
		rgb=(0.824940, 0.884720, 0.106217)
		rgb=(0.835270, 0.886029, 0.102646)
		rgb=(0.845561, 0.887322, 0.099702)
		rgb=(0.855810, 0.888601, 0.097452)
		rgb=(0.866013, 0.889868, 0.095953)
		rgb=(0.876168, 0.891125, 0.095250)
		rgb=(0.886271, 0.892374, 0.095374)
		rgb=(0.896320, 0.893616, 0.096335)
		rgb=(0.906311, 0.894855, 0.098125)
		rgb=(0.916242, 0.896091, 0.100717)
		rgb=(0.926106, 0.897330, 0.104071)
		rgb=(0.935904, 0.898570, 0.108131)
		rgb=(0.945636, 0.899815, 0.112838)
		rgb=(0.955300, 0.901065, 0.118128)
		rgb=(0.964894, 0.902323, 0.123941)
		rgb=(0.974417, 0.903590, 0.130215)
		rgb=(0.983868, 0.904867, 0.136897)
		rgb=(0.993248, 0.906157, 0.143936)
	}
}

\journal{Journal of Computational Physics}

\begin{document}

\begin{frontmatter}

  \title{Entropy stable high-order discontinuous Galerkin spectral-element methods on curvilinear, hybrid meshes}

  \author[label1]{Jens Keim\corref{cor}\fnref{fn}} %
  \author[label1]{Anna Schwarz\corref{cor}\fnref{fn}} %
  \author[label1]{Patrick Kopper} %
  \author[label1]{Marcel Blind} %
  \author[label2]{Christian Rohde} %
  \author[label1]{Andrea Beck} %

  \cortext[cor]{Corresponding author}
  \fntext[fn]{J. Keim and A. Schwarz share first authorship.}
  \affiliation[label1]{organization={Institute of Aerodynamics and Gas Dynamics, University of Stuttgart},%
    addressline={Wankelstr. 3},
    city={Stuttgart},
    postcode={70563},
    country={Germany}}
  \affiliation[label2]{organization={Institute of Applied Analysis and Numerical Simulation, University of Stuttgart},%
    addressline={Pfaffenwaldring 57},
    city={Stuttgart},
    postcode={70569},
    country={Germany}}

  \begin{abstract}
    Hyperbolic-parabolic partial differential equations are widely used for the modeling of complex, multiscale problems in science and engineering, e.g., turbulent flows.
    High-order methods such as the discontinuous Galerkin (\DG) scheme are attractive candidates for their numerical approximation.
    However, high-order methods are prone to instabilities in the presence of underresolved flow features.
    A popular counter measure to stabilize \DG~methods is the use of entropy-stable formulations based on summation-by-parts (\SBP) operators.
    While \textit{multidimensional} \SBP~operators enabled the use of entropy-stable formulations on non-hexahedral elements,
    the efficient construction of such operators is challenging.
    Therefore, it is advantageous to employ a tensor-product-based approach such as the discontinuous Galerkin spectral element method (\DGSEM), which, in its original form, is
    however limited to hexahedral elements.
    \newline
    The present paper aims to construct a robust and efficient entropy-stable \DGSE~scheme of arbitrary order on heterogeneous,
    curvilinear grids composed of triangular and quadrilateral elements or hexahedral, prismatic, tetrahedral and pyramid elements.
    To the author's knowledge, with the exception of hexahedral and quadrilateral elements, entropy-stable \DGSE~operators
    have been constructed exclusively for tetrahedral and triangular meshes.
    The extension of the \DGSEM~to more complex element shapes is achieved by means of a collapsed coordinate transformation.
    Legendre--Gauss quadrature nodes are employed as collocation points in conjunction with a generalized SBP operator and
    entropy-projected variables.
    The purely hyperbolic operator is extended to hyperbolic-parabolic problems by the use of a lifting procedure.
    To circumvent the penalizing time step restriction imposed by the collapsing, modal rather than nodal degrees of freedom are
    evolved in time, thereby relying on a memory-efficient weight-adjusted approximation to the inverse of the mass matrix.
    Essential properties of the proposed numerical scheme including free-stream preservation, polynomial and grid convergence as
    well as entropy conservation / stability are verified.
    Finally, with the flow around the common research model, the applicability of the presented method to real-world problems is demonstrated.
  \end{abstract}


\end{frontmatter}

\section{Introduction}%
\label{sec:introduction}

Nonlinear hyperbolic or hyperbolic-parabolic conservation laws are widely used for the modeling of a variety of physical processes in science and engineering, e.g., in turbulent and electromagnetic flows or climate forecast, only to name a few examples.
Intrinsic features of nonlinear hyperbolic partial differential equations (\PDEs) are the generation of multiscale phenomena and the breakdown of classical solutions.
The latter eventually leads to the occurrence of discontinuities.
This multitude of flow features poses several challenges for the design of adequate numerical schemes.
First, it is imperative that they exhibit both high efficiency and scalability across massively parallel high-performance computing systems due to the immense amount of degrees of freedom involved.
Secondly, the necessity of accuracy in conjunction with robustness is paramount.
Finally, concerning application-oriented simulations, contemporary numerical schemes must include geometric flexibility, thereby facilitating the use of heterogeneous meshes.
\newline
A class of numerical methods able to meet a vast majority of these requirements are high-order methods capable of dealing with unstructured meshes.
A general superiority of high-order methods over their low-order counterparts are their excellent dispersion and dissipation properties for smooth solutions~\cite{Gassner2011}.
Among the high-order methods, the discontinuous Galerkin method has proven to be particularly well suited to fulfill these requirements.
The inherent data locality of \DG~schemes has demonstrated to provide superior performance on modern hardware and parallel systems~\cite{Blind2024a,Kempf2024}.
Moreover, DG schemes offer high flexibility in element sizes and shapes, allow for adaptive mesh refinement, and permit a local adjustment of the order of accuracy~\cite{Mavriplis1994, Mossier2022}.
Consequently, the \DG~and derived methods such as the tensor-product based discontinuous Galerkin spectral element method gained great popularity over recent decades, manifesting in a plethora of highly efficient code frameworks~\cite{Cantwell2015,Witherden2015,Krais2019,Schlottkelakemper2020,Parsani2021,Ferrer2023,Kempf2024}.
\newline
A general drawback of DG methods, similar to other high-order methods, is their lack of robustness in the presence of discontinuities (Gibbs phenomenon) or due to strongly non-linear flux functions (aliasing).
Hence, additional stabilization techniques are required to prevent the generation of physically inadmissible states which might cause the crash of a simulation or even worse, the divergence to a physically admissible but wrong solution.
Well-established approaches to stabilize high-order methods are, e.g., artificial viscosity~\cite{Persson2012}, modal filtering~\cite{Dzanic2022}, slope limiting~\cite{Kuzmin2021}, polynomial de-aliasing~\cite{Gassner2012}, and total variation diminishing or total variation bounded finite-volume schemes~\cite{Sonntag14,Dumbser2016,Mossier2022,Huerta2012,Hennemann2021}.
However, these stabilization techniques often come along with heuristic tuning parameters and lack any rigorous proofs for stability.
Conversely, entropy-stable schemes guarantee physically admissible solutions over time.
Their only requirements are a stable time step restriction, the presence of a convex entropy and an associated entropy/entropy flux pair, and the fulfillment of certain physical admissibility criteria, e.g., the positivity of density and pressure in case of the Euler equations~\cite{Chan2018}.
\newline
Discretely entropy conservative and entropy stable schemes for hyperbolic systems of conservation laws were first developed for low-order finite volume methods by \citet{Tadmor1987}, utilizing specifically designed two-point fluxes which conserve a particular mathematical entropy.
This approach was subsequently extended to fully-discrete and semi-discrete higher-order schemes on periodic domains and structured grids~\cite{LeFloch2002}.
The development of high-order entropy-stable discretizations based on finite-differences~\cite{Fisher2013} and spectral collocation
schemes~\cite{Carpenter2014} for the Euler and Navier-Stokes equations constituted the basis for contemporary entropy-stable formulations.
For this, the authors relied on the combination of entropy conservative two-point flux functions, summation-by-parts operators, and a diagonal mass matrix.
\citet{Gassner2013} was the first to recognize that the matrix operators of a \DGSE~scheme on Legendre--Gauss--Lobatto nodes are SBP operators.
This enabled the construction of entropy-stable \DGSE~discretizations, first for the Burgers equation~\cite{Gassner2013} and later for systems of conservation laws~\cite{Carpenter2014,Gassner2016}.
The extension to a more flexible choice of compute and quadrature nodes was achieved by \citet{DelReyFernandez2014} based on a \textit{generalized SBP property} and further applied to non-tensor-product simplex elements based on the use of \textit{multidimensional SPB operators}~\cite{Hicken2016}.
These developments initiated a strong interest and further research on entropy-stable \DG~schemes on triangular
and tetrahedral elements~\cite{Chen2017,Chan2017,Crean2018,Chan2019}, see~\cite{TianhengChen2020} for a recent overview.
\newline
A general drawback of multidimensional operators and especially \textit{multidimensional SBP operators} is the significantly increasing computational costs which scale asymptotically with $\mathcal{O} (\ppn^{2 d})$, where $\ppn$ and $d$ indicate the polynomial degree and the number of spatial dimensions, respectively.
Conversely, tensor-product operators used in \DGSEM~are amenable to sum-factorization techniques~\cite{Orszag1980} and scale asymptotically with $\mathcal{O} (\ppn^{d+1})$.
Additionally, \textit{multidimensional SBP operators} require the evaluation of two-point fluxes between all pairs of quadrature nodes compared to the line-wise procedure utilized in tensor-product formulations.
In view of the aforementioned drawbacks, in conjunction with the challenges in deriving suitable quadrature rules for high-polynomial degrees~\cite{Worku2024}, \textit{multidimensional SBP operators} are only rarely employed for polynomial degrees greater than four or five~\cite{Montoya2024b}.
However, due to the excellent exponential convergence properties, polynomial degrees of $N \geq 5$ are favorable and nowadays common practice in turbulence simulations even for application-oriented problems, possibly including shocks~\cite{Durrwachter2021,Blind2023,Blind2024b}.
\newline
In a recent series of papers, \citeauthor{Montoya2022}~\cite{Montoya2022,Montoya2024a,Montoya2024b} exploited the favorable properties of tensor-product operators on geometrically flexible element shapes by utilizing efficient collocated tensor-product SBP operators on collapsed coordinates for curved triangular and tetrahedral elements using Legendre--Gauss-Lobatto and Legendre--Gauss--Radau quadrature nodes.
For this, the authors followed the theoretical framework of \citet{Dubiner1991}, which enabled the extension of the \DGSEM~to non-hexahedral meshes~\cite{Lomtev1999,Warburton1999,Kirby2000,Chan2016} through the use of a \textit{Duffy transformation}~\cite{Duffy1982}.
This tensor-product SBP operator has been demonstrated to achieve an asymptotically optimal $\mathcal{O} ( \ppn^{\dd+1})$ complexity in terms of floating-point operations, similar to the \DGSEM~on quadrilateral and hexahedral elements.
Furthermore, through time evolution of modal rather than nodal \DOF, the authors were able to avoid the penalizing time step restriction commonly observed for collapsed coordinate discretizations~\cite{Dubiner1991}.
Here, the authors exploited the \enquote{warped} tensor-product structure of the Proriol--Koornwinder--Dubiner (\PKD) orthogonal modal polynomial basis~\cite{Proriol1957,Koornwinder1975,Dubiner1991} in conjunction with a weight-adjusted approximation of the inverse of the curvilinear mass matrix proposed by~\cite{Chan2017}.
The combination enabled an efficient and high-order accurate numerical scheme for arbitrary polynomial degrees without the aforementioned drawbacks of \textit{multidimensional SBP operators}.
\newline
It is important to highlight that aside from the previously described collapsed coordinates transformation, a variety of additional approaches are available in the literature which also enable the use of (entropy-stable) \DGSEM~for flow problems with complex geometries.
Examples are immersed boundary schemes~\cite{Ferrer2023,Kemm2020}, cut-cell and extended \DG~based approaches~\cite{Kummer2021,Henneaux2023,May2022,Taylor2024}, or the splitting of polymorphic into hexahedral elements~\cite{Parsani2021,Worku2025}.
All the aforementioned methods have their validity for existence, however, the performance, accuracy, and applicability of each may be dependent upon the specific test case.
\newline
The present paper aims to extend the conceptual ideas outlined in~\cite{Montoya2024a} to prismatic and pyramidal elements and
eventually construct a robust and efficient entropy-stable \DGSE~scheme on mixed element meshes composed of hexahedral, prismatic, tetrahedral and pyramid elements.
In contrast to~\cite{Montoya2024a}, a \DGSEM~with a Legendre--Gauss quadrature is utilized in combination with a generalized SBP
operator and entropy-projected variables.
A Legendre--Gauss quadrature is chosen due to the increased integration accuracy compared to Legendre--Gauss--Lobatto or
Legendre--Gauss--Radau nodes~\cite{Chan2022a}.
Additionally, as pointed out by~\cite{Schwarz2025}, \DGSEM~with Legendre--Gauss nodes provides a more efficient scheme for
underresolved turbulent flows compared to Legendre--Gauss--Lobatto nodes, if a similar level of accuracy has to be achieved.
To the authors' best knowledge, this is the first time an entropy-stable \DGSEM~has been formulated for hybrid grids featuring hexahedral, tetrahedral, prismatic, and pyramidal elements.
To summarize, the contributions of the present work are:
\begin{itemize}
  \item An entropy conservative / entropy stable \DGSE~operator of arbitrary order on curvilinear meshes is constructed based on the diagonal-norm SBP operator on the reference hexahedron, tetrahedron, prism, and pyramid using Legendre--Gauss quadrature. In two dimensions, the SBP operator is defined on the reference square and triangle.
  \item The hyperbolic operator is extended to hyperbolic-parabolic problems by approximating viscous fluxes with the BR1 (lifting) scheme~\cite{Bassi1997}. This is achieved similarly to the hyperbolic part by utilizing the collapsed coordinates transformation.
  \item A modal-based formulation of the proposed \DGSEM~is advanced in time, but a nodal-based spatial integration with more quadrature points than modal degrees of freedom is utilized \cite{Gassner2009}.
    For this, the \enquote{warped} tensor-product structure of the orthogonal modal polynomial basis is exploited in conjunction with a weight-adjusted approximation of the inverse of the curvilinear mass matrix.
    This offers the potential for a reduction in computational cost by exploiting tensor-product properties, while simultaneously avoiding excessively small time steps \cite{Montoya2024a}.
  \item The free-stream preservation, convergence properties, and robustness of the proposed scheme are demonstrated.
    For the latter, the weakly compressible Taylor--Green vortex serves as a test case.
  \item The applicability of the proposed scheme to more complex, large-scale problems is demonstrated using the flow around the common research model (CRM) as an example.
\end{itemize}
In summary, this work facilitates the high-order simulation of more complex, (three-dimensional) large-scale applications that would be infeasible to mesh using only hexahedral elements.
Combined with the entropy-stable tensor-product based \DGSEM~with a diagonal-norm SBP operator on polyhedral elements, this provides a robust and stable numerical scheme for underresolved turbulent flows of arbitrary order on curvilinear hybrid meshes.
To the authors knowledge, this is the first time a \DGSEM-based entropy-stable split formulation has been proposed for prismatic and pyramid elements.
In contrast to known approaches in the literature \cite{Worku2025,Parsani2021}, both based on the splitting of tetrahedral or polyhedral elements into hexahedrals, the present approach eschews the unnecessary generation of additional elements through this process of element splitting.
Moreover, zonal flow oriented hexahedral based boundary grid refinements can be easily realized and automatically connected to the overall coarse grid.

The outline of this paper is as follows:
Following this introduction, the theoretical foundation of this paper is presented in~\cref{sec:theory}, starting from the governing equations including entropy variables and progressing to the collapsed coordinate transformation.
The entropy stable DGSE scheme on mixed elements is presented in~\cref{sec:methods}, followed by a detailed validation of the implementation by numerical experiments in~\cref{sec:validation}.
This paper concludes with a brief summary in~\cref{sec:conclusion}.

\section{Fundamentals}%
\label{sec:theory}

This section starts with the theoretical foundation of this paper.
First, the considered governing equations, and the corresponding entropy variables are presented.
Since the focus of this work is the derivation of an efficient DGSE method on hybrid meshes, two steps are necessary:
\begin{enumerate*}[label={\alph*)}]
  \item A transformation of the governing equations from the computational domain $\Omega$ in physical space to the polytopal reference space $\refraumcol$ composed of
unit hexahedral, prisms, pyramids, or tetrahedrons, given by the inverse of the mapping $\mapping : \refraumcol \to \Omega, \refpos
\mapsto \mathbf{x}$.
  \item The transformation from the polytopal reference space $\refraumcol$ to the unit hexahedral $\refraum$ using the inverse of the mapping
$\mappingcol : \refraum = [-1,1]^d \to \refraumcol, \refposquad \mapsto \refpos$, called collapsed coordinates transformation.
\end{enumerate*}
An illustration is given in~\cref{fig:theory:mapping}.
As such, the second part of this section provides a concise overview of the mappings from the polytopal reference space to the unit
hexahedral.
It is imperative to acknowledge that the general 3D case is considered throughout this work.

\begin{figure}[htb!]
  \includegraphics[width=0.99\linewidth]{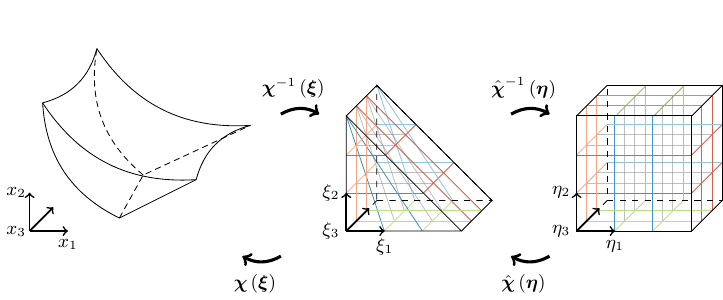}
  \caption{Schematic sketch of the transformation from the physical to the polytopal reference element and the unit
  hexahedral element (from left to right).}%
  \label{fig:theory:mapping}
\end{figure}

\subsection{Governing Equations}%
\label{sec:theory:governingeq}

The governing equations considered in this work are the compressible Navier--Stokes equations (NSE), written as
\begin{align}
  \lr{\cons}_t + \gradientX \cdot \fphys(\cons,\nabla \cons) = \Null, \hspace{0.5cm} \ \fphys(\cons,\nabla \cons) = \Fc - \Fv,
  \label{eq:theory:euler}
\end{align}
where $\cons=\arr{\rho, \rho \vel, \rho e} \in \R^{\nvar}$, $\nvar=d+2$, is the vector of conserved variables, consisting of the density $\rho$, the velocity
vector $\vel=\arr{\vel[1],...,\vel[d]} \in \R^d$ and the total energy $e$ per unit mass, and $\fphys \in \R^{d\times\nvar}$ are the
physical fluxes, comprised of the convective flux $\Fc$ and the viscous flux $\Fv$, given as
\begin{align}
  \Fc =& \arr{
    \makebox[\widthof{$\rho \vel$}][l]{$\rho \vel$},
    \makebox[\widthof{$\rho \vel \otimes \vel + \p \mathbf{I}$}][l]{$\rho \vel \otimes \vel + \p \mathbf{I}$},
    \makebox[\widthof{$\stress \cdot \vel + \conductivity \nabla T$}][l]{$(\rho e + \p) \vel$}
    }, \\
  \Fv =& \arr{
    \makebox[\widthof{$\rho \vel$}][l]{$0$},
    \makebox[\widthof{$\rho \vel \otimes \vel + \p \mathbf{I}$}][l]{$\stress$},
    \makebox[\widthof{$\stress \cdot \vel + \conductivity \nabla T$}][l]{$\stress \cdot \vel + \conductivity \nabla T$}
    }.
\end{align}
with the unit tensor $\mathbf{I} \in \R^{3 \times 3}$, the temperature $T$, the thermal conductivity $\conductivity$, and the pressure for a perfect gas $p=(\gamma-1)(\rho e- 0.5 \rho (\vel \cdot \vel))$ with $\gamma=1.4$.
Following Stokes' hypothesis, which assumes that the bulk viscosity is zero, the viscous stress tensor reduces to $\stress =
\dynvisc (\lr{\nabla \vel}^\transpose + \nabla \vel - \frac{2}{3} \lr{\nabla \cdot \vel}\mathbf{I}) \in \R^{3 \times 3}$ for a Newtonian fluid, where $\dynvisc$ denotes the dynamic viscosity.
The heat flux is modeled according to Fourier's hypothesis with $\conductivity = \frac{c_p \dynvisc}{\Pr}$, the Prandtl number $\Pr$, and the specific heat at constant pressure $c_p$ of ambient air.

\subsubsection{Entropy Variables}
The NSE are equipped with a convex entropy/entropy flux pair, given as
\begin{align}
  \lr{\entropy,\fentropyref} = \lr{-\frac{\rho \fphysentropyref}{\gamma-1},-\frac{\rho \fphysentropyref \vel}{\gamma-1}}
\end{align}
with the thermodynamic entropy $\fphysentropyref=\log(p\rho^{-\gamma})$.
The corresponding entropy variables are the Jacobian of the mathematical entropy with respect to the conservative variables, given
as
\begin{align}
  \entropyvar(\cons) \coloneq \fracp{v}{\cons} = \arr{\frac{\gamma-\fphysentropyref}{\gamma-1}-\frac{\rho}{2p} \abs{\vel}^2,\frac{\rho
  \vel}{p},-\frac{\rho}{p}}.
\end{align}
The conservative variables can be expressed in terms of the entropy variables as
\begin{align}
  \cons(\entropyvartilde) = \arr{- (\rho \epsilon) \entropyvartilde[d+2], (\rho \epsilon) \entropyvartilde_{2:d+1}, (\rho \epsilon)
  \lr{1-\frac{\abs{\entropyvartilde_{2:d+1}}^2}{2\entropyvartilde[d+2]}}}
\end{align}
with $\entropyvartilde=(\gamma-1)\entropyvar$ and the specific internal energy $\rho \epsilon$ defined by the entropy variables as
\begin{align}
  \rho \epsilon = \lr{\frac{(\gamma-1)}{(-\entropyvartilde[d+2])^\gamma}}^{1/(\gamma-1)} \exp{\lr{\frac{-\fphysentropyref}{\gamma-1}}}, \quad
  \fphysentropyref = \gamma - \entropyvartilde[1] + \frac{\abs{\entropyvartilde_{2:d+1}}^2}{2 \entropyvartilde[d+2]}.
\end{align}
For a well defined entropy, positivity of density and pressure has to be guaranteed.
Multiplying~\cref{eq:theory:euler} with the entropy variables and integrating over $\Omega$ yields the entropy inequality for the compressible NSE,
\begin{align}
  \int_{\Omega}{\entropy_t}~dx \leq \int_{\partial \Omega} \lr{\entropyvar^\transpose \fphys_v - \fentropyref} \cdot \normalvecref~dS_{\Omega},
\end{align}
see, e.g.,~\cite{Dalcin2019,Chan2022b} for more details.

\subsection{Collapsed Coordinates Transformation}%
\label{sec:theory:collcoord}

Starting from the computational domain $\Omega$, the governing equations can be transformed to the polytopal reference space using
$\mapping^{-1}$.
The Jacobian matrix of this mapping is defined as $\Jm = \lr{\gradientXI \mapping} \in \Rdd$ with its corresponding determinant $\J=\det\Jm \in \R$.
This results in
\begin{align}
  \J \lr{\cons}_t + \gradientXI \cdot \fphysref = \Null, \hspace{0.5cm}  \fphysref = \M^\transpose \fphys,
\end{align}
with the contravariant fluxes $\fphysref$, the metric terms $\M=\lr{\adj \Jm}^\transpose \otimes \Iunit_{\nvar}$, and the identity matrix $\Iunit_{\nvar} \in
\R^{\nvar\times\nvar}$.
The term $\adj (\cdot)$ denotes the adjoint matrix.
Subsequently, the polytopal reference space is transformed to the unit hexahedral to enable the derivation of the DGSEM.
For this, the mapping $\mappingcol : \refraum = [-1,1]^d \to \refraumcol, \refposquad
\mapsto \refpos$, from the unit hexahedral to the polytopal reference space composed of the unit prism, pyramid, or tetrahedron (for $d=3$) and from the unit square to the unit triangle (for $d=2$) can be defined.
The coordinates in $\refraum$ and $\refraumcol$ are denoted as $\refposquad = \arr{\refposquad[1],\refposquad[2],\refposquad[3]}$ and $\refpos = \arr{\refpos[1],\refpos[2],\refpos[3]}$, respectively.
The mapping $\mappingcol$ is equipped with its Jacobian $\Jmref = \gradientETA \mappingcol \in \R^{d \times d}$, the corresponding inverse $\Jmref^{-1} = (\gradientETA
\hat{\mapping})^{-1} = \Jref^{-1} \adj \Jmref$, and the determinant of the Jacobian $\Jref = \det \Jmref \in \R$.
The boundary of the reference element $\partial \refraumcol$ is partitioned into $\Nf$ faces $\{\Gamma^{\ifa}\}_{\ifa \in
  \{1:\Nf\}}, \Gamma^{\ifa} \subset \partial \refraumcol$ each equipped with a constant unit outward normal vector
    $\normalvecref^{(\ifa)}$.

In order to calculate the metric terms in the polytopal reference space, a set of corresponding basis functions has to be
defined.
For this, the polytopal reference space is equipped with a set of normalized Proriol--Koornwinder--Dubiner (PKD) orthonormal basis functions that are recursively constructed using normalized Jacobi polynomials $\mathcal{P}^{(\alpha,\beta)} \in \mathbb{P}^{\ppn}(\refraumcol)$ which are orthogonal with respect to the weight functions $(1-\refposquad)^{\alpha}(1-\refposquad)^{\beta}, \ \alpha,\beta > -1$, see, e.g.,~\cite{Hesthaven2008} for further details.
Here, $\mathbb{P}^{\ppn}(\refraumcol)$ denotes the space of polynomials of degree $\ppn$ on the polytopal reference element,
$\refraumcol$, with $M(\ppn)$ as the number of dimensions of this space.

In the following, the polytopal reference space for each element type is introduced, including the corresponding basis functions,
Jacobian, outward-pointing normal vectors and quadrature weights.
The sequel is a collection of material that originates from~\cite{Karniadakis2005,Bergot2013,Chan2017}.

\subsubsection{Volume and Surface Quadrature}
The mapping $\mappingcol$ allows to approximate integrals on the polytopal reference element by a tensor-product quadrature rule on the unit cube with $\ppn+1$ distinct nodes $\{\refposquad_{i}\}_{i=0}^{\ppn}$ and non-zero weights $\{\w_i\}_{i=0}^{\ppn}$ in each direction, resulting in a total of $\nq=(\ppn+1)^d$ quadrature nodes $\refposquadq \in \R^{d \times \nq}$ in the volume.
For an arbitrary function $f(\refpos)$, this leads to
\begin{align}
  \int_{\refraumcol} f(\refpos)~d\refpos = \intE{f(\refposquad) \Jref(\refposquad)}
  \approx \sum\limits_{i,j,k=0}^{\ppn} f(\refposquad_{ijk}) \Jref(\refposquad_{ijk}) \w_{ijk}
\end{align}
with $\w_{ijk} = \w_i \w_j \w_k$.
Surface integrals can be similarly approximated as
\begin{align}
  \int_{\partial \refraumcol} f(\refpos)~d\hat{S} = \intpE{f(\refposquad) \Jref^{(\ifa)}(\refposquad)}
  \approx \sum\limits_{i,j=0}^{\ppn} f(\refposquad_{ij}) \Jref^{(\ifa)}(\refposquad_{ij}) \w_{i} \w_{j}
\end{align}
with the corresponding determinant of the Jacobian on the surface $\Jref^{(\ifa)}(\refposquad)$.
This results in a total of $\nf = (\ppn+1)^{d-1}$ distinct quadrature nodes on the surface $\refposquad^{(\ifa)} \in \R^{d \times \nf}$.
\subsubsection{Prismatic and Triangular Elements}
The reference domain of the unit prism is given by
\begin{align}
  \refraumcol = \{\refpos \in [-1,1]^3 : \refpos[1] + \refpos[2] \leq 0\}
\end{align}
with $M(\ppn) = \frac{1}{2} (\ppn+1)^2(\ppn+2)$ orthonormal (w.r.t the corresponding weights) basis functions
\begin{align}
  \testfunc_{ijk}(\refpos) = \sqrt{2} \mathcal{P}_i^{(0,0)}(\refposquad[1]) \mathcal{P}_j^{(2i+1,0)}(\refposquad[2])
  \mathcal{P}_k^{(0,0)}(\refposquad[3]) \ (1-\refposquad[2])^i, \nonumber \\
  \forall (i,j,k) \geq 0 \wedge i+j \leq \ppn \wedge k \leq \ppn.
\end{align}
The mapping $\mappingcol: \refraum \to \refraumcol$ from the unit cube to the reference prism and its inverse are written as
\begin{align}
  \mappingcol(\refposquad)  &= \arr{\frac{(1+\refposquad[1])(1-\refposquad[2])}{2} - 1, \refposquad[2], \refposquad[3]}, \\
  \mappingcol^{-1}(\refpos) &= \arr{2 \frac{1+\refpos[1]}{1-\refpos[2]} - 1, \refpos[2], \refpos[3]: \refpos[2] \ne 1}.
\end{align}
The adjoint $\adj \Jmref$ and the inverse of the determinant of the Jacobian $\Jref^{-1}$ of the mapping from the unit cube to the unit prism are defined as
\begin{align}
  \adj \Jmref(\refpos) =
  \adj \Jmref(\mappingcol(\refposquad)) =
  \left(\begin{matrix}
    1                            & 0                            & 0                            \\
    \frac{(1+\refposquad[1])}{2} & \frac{(1-\refposquad[2])}{2} & 0                            \\
    0                            & 0                            & \frac{(1-\refposquad[2])}{2} \\
  \end{matrix}\right)
\end{align}
and $\Jref^{-1} =\frac{2}{1-\refposquad[2]}$, respectively.
The normalized normal vectors $\normalvecref$ in the polytopal reference space for each of the five faces are given as
\begin{align}
  \normalvecref^{(\ifa)} = \begin{cases}
      \arr{0,0,-1} &:~\Gamma^{\ifa}=\Gamma^{1} \coloneq \{\refpos \in \refraum: \refpos[3]=-1\}, \\
      \arr{0,-1,0} &:~\Gamma^{\ifa}=\Gamma^{2} \coloneq \{\refpos \in \refraum: \refpos[2]=-1\}, \\
      \arr{\nicefrac{1}{\sqrt{2}},\nicefrac{1}{\sqrt{2}},0} &:~\Gamma^{\ifa}=\Gamma^{3} \coloneq \{\refpos \in \refraum:
        \refpos[1]+\refpos[2]=0\}, \\
      \arr{-1,0,0} &:~\Gamma^{\ifa}=\Gamma^{4} \coloneq \{\refpos \in \refraum: \refpos[1]=-1\}, \\
       \arr{0,0,1} &:~\Gamma^{\ifa}=\Gamma^{5} \coloneq \{\refpos \in \refraum: \refpos[3]=1\}.
    \end{cases}
\end{align}
The corresponding determinant of the Jacobian on the surface $\Jref^{(\ifa)}(\refposquad_{ij}), \ i,j=0,...,\ppn,$ is defined as
\begin{align}
  \Jref^{(\ifa)}(\refposquad_{ij}) =
  \begin{cases}
    \frac{(1-\refposquad[2]_j)}{2} &: \ifa = \{1,5\} \ (\text{triangular side}), \\
    \sqrt{2}                     &: \ifa = 3, \\
    1                              &: \text{otherwise},
  \end{cases}
\end{align}
with $\Jref^{(\ifa)} \in \R^{(\ppn+1) \times (\ppn+1)}$.
The 2D equivalent to the prismatic element is a triangle. As such, in 2D, the corresponding mapping reduces to the reference
triangle, given as
\begin{align}
  \refraumcol = \{\refpos \in [-1,1]^2, \refpos[1] + \refpos[2] \leq 0\}
  \label{eq:refraum_tria}
\end{align}
with the dimension of the ansatz space $M(\ppn) = \frac{1}{2} (\ppn+1)(\ppn+2)$, i.e., the third direction and its corresponding faces $\Gamma^1$ and $\Gamma^5$ are dropped.

\subsubsection{Pyramid Elements}

The reference domain of the biunit right pyramid is given by
\begin{align}
  \refraumcol = \{\refpos \in [-1,1]^3 : \refpos[1] + \refpos[3] \leq 0; \refpos[2] + \refpos[3] \leq 0\}
\end{align}
with $M(\ppn)=\frac{1}{6} (\ppn+1)(\ppn+2)(2\ppn+3)$ orthonormal (w.r.t the corresponding weights) polynomial (semi-nodal) basis functions
\begin{align}
  \testfunc_{ijk}(\refpos) = \frac{1}{C_{nk}} \mathcal{P}_i^{(0,0)}(\refposquad[1]) \mathcal{P}_j^{(0,0)}(\refposquad[2])
  \mathcal{P}_k^{(2(\ppn-k)+3,0)}(\refposquad[3]) \ &\lr{\frac{1-\refposquad[3]}{2}}^{\ppn-k}, \nonumber \\
  \forall (i,j,k) \geq 0 \wedge i+k \leq \ppn \wedge j+k \leq \ppn. &
\end{align}
with $C_{nk}=\frac{\ppn+2}{2^{2(\ppn-k)+2}(2(\ppn-k)+3)}$.
The mapping $\mappingcol: \refraum \to \refraumcol$ from the unit cube to the unit pyramid and its inverse are written as
\begin{align}
  \mappingcol(\refposquad)  &= \arr{\frac{(1+\refposquad[1])(1-\refposquad[3])}{2} - 1, \frac{(1+\refposquad[2])(1-\refposquad[3])}{2} - 1, \refposquad[3]}, \\
  \mappingcol^{-1}(\refpos) &= \arr{2 \frac{1+\refpos[1]}{1-\refpos[3]} - 1, 2 \frac{1+\refpos[2]}{1-\refpos[3]} - 1, \refpos[3]:
  \refpos[3] \neq 1}.
\end{align}
The adjoint $\adj \Jmref$ and the inverse of the determinant of the Jacobian $\Jref^{-1}$ from the mapping of the unit cube to the unit
pyramid are defined as
\begin{align}
  \adj \Jmref(\mappingcol(\refposquad)) =
  \left(\begin{matrix}
    \frac{(1-\refposquad[3])}{2}                   & 0                                      &   0                        \\
    0                                              & \frac{(1-\refposquad[3])}{2}           & 0                          \\
    \frac{(1+\refposquad[1])(1-\refposquad[3])}{4} & \frac{(1+\refposquad[2])(1-\refposquad[3])}{4} & \frac{(1-\refposquad[3])^2}{4} \\
  \end{matrix}\right)
\end{align}
and $\Jref^{-1} =\frac{4}{(1-\refposquad[3])^2}$, respectively.
The normalized normal vectors $\normalvecref$ in the polytopal reference space for each of the five faces are given as
\begin{align}
  \normalvecref^{(\ifa)} = \begin{cases}
      \arr{0,0,-1} &:~\Gamma^{\ifa}=\Gamma^{1} \coloneq \{\refpos \in \refraum: \refpos[3]=-1\}, \\
      \arr{0,-1,0} &:~\Gamma^{\ifa}=\Gamma^{2} \coloneq \{\refpos \in \refraum: \refpos[2]=-1\}, \\
      \arr{\nicefrac{1}{\sqrt{2}},0,\nicefrac{1}{\sqrt{2}}} &:~\Gamma^{\ifa}=\Gamma^{3} \coloneq \{\refpos \in \refraum:
        \refpos[1]+\refpos[3]=0\}, \\
      \arr{-1,0,0} &:~\Gamma^{\ifa}=\Gamma^{4} \coloneq \{\refpos \in \refraum: \refpos[1]=-1\}, \\
      \arr{0,\nicefrac{1}{\sqrt{2}},\nicefrac{1}{\sqrt{2}}} &:~\Gamma^{\ifa}=\Gamma^{5} \coloneq \{\refpos \in \refraum:
        \refpos[2]+\refpos[3]=0\}.
    \end{cases}
\end{align}
The corresponding determinant of the Jacobian on the surface $\Jref^{(\ifa)}(\refposquad_{ij}), \ i,j=0,...,\ppn,$ is defined as
\begin{align}
  \Jref^{(\ifa)}(\refposquad_{ij}) =
  \begin{cases}
    \frac{(1-\refposquad[3]_j)}{2} &: \ifa \in \{2,4\} \ (\text{triangular side}), \\
    \sqrt{2} \frac{(1-\refposquad[3]_j)}{2} &: \ifa \in \{3,5\} \ (\text{triangular side}), \\
    1 &: \text{otherwise},
  \end{cases}
\end{align}
with $\Jref^{(\ifa)} \in \R^{(\ppn+1) \times (\ppn+1)}$.
It has to be noted that there is no 2D equivalent to the pyramid.

\subsubsection{Tetrahedral Elements}

The reference domain of the unit tetrahedron is given by
\begin{align}
  \refraumcol = \{\refpos \in [-1,1]^3 : \refpos[1] + \refpos[2] + \refpos[3] \leq 0\}
\end{align}
with $M(\ppn)=\frac{1}{6} (\ppn+1)(\ppn+2)(\ppn+3)$ orthonormal (w.r.t the corresponding weights) basis functions
\begin{align}
  \testfunc_{ijk}(\refpos) = 2 \sqrt{2} \mathcal{P}_i^{(0,0)}(\refposquad[1]) \mathcal{P}_j^{(2i+1,0)}(\refposquad[2])
  \mathcal{P}_k^{(2(i+j)+2,0)}(\refposquad[3]) \ (1-\refposquad[2])^i &(1-\refposquad[3])^{i+j}, \nonumber \\
  \forall (i,j,k) \geq 0 \wedge i+j+k \leq \ppn. &
\end{align}
The mapping $\mappingcol: \refraum \to \refraumcol$ from the unit cube to the unit tetrahedron and its inverse are written as
\begin{align}
  \mappingcol(\refposquad)  &= \arr{\frac{(1+\refposquad[1])(1-\refposquad[2])(1-\refposquad[3])}{4} - 1, \frac{(1+\refposquad[2])(1-\refposquad[3])}{2} - 1, \refposquad[3]}, \\
  \mappingcol^{-1}(\refpos) &= \arr{2 \frac{1+\refpos[1]}{-\refpos[2]-\refpos[3]} - 1, 2 \frac{1+\refpos[2]}{1-\refpos[3]} - 1, \refpos[3]:
  \refpos[3] \neq 1 \wedge (\refpos[2]+\refpos[3]) \neq 0}.
\end{align}
The adjoint $\adj \Jmref$ and the inverse of the determinant of the Jacobian $\Jref^{-1}$ of the mapping from the unit cube to the unit tetrahedron are defined as
\begin{align}
  \adj \Jmref(\mappingcol(\refposquad)) =
  \left(\begin{matrix}
    \frac{(1-\refposquad[3])}{2}               & 0                                            & 0                                 \\
    \frac{(1+\refposquad[1])(1-\refposquad[3])}{4} & \frac{(1+\refposquad[2])(1-\refposquad[3])}{4}                   & 0                                \\
    \frac{(1+\refposquad[1])(1-\refposquad[3])}{4} & \frac{(1+\refposquad[2])(1-\refposquad[2])(1-\refposquad[3])}{8} &
    \frac{(1-\refposquad[2])(1-\refposquad[3])^2}{8} \\
  \end{matrix}\right)
\end{align}
and $\Jref^{-1} =\frac{8}{(1-\refposquad[2])(1-\refposquad[3])^2}$, respectively.
The normalized normal vectors $\normalvecref$ in the polytopal reference space for each of the four faces are given as
\begin{align}
  \normalvecref^{(\ifa)} = \begin{cases}
      \arr{0,0,-1} &:~\Gamma^{\ifa}=\Gamma^{1} \coloneq \{\refpos \in \refraum: \refpos[3]=-1\}, \\
      \arr{0,-1,0} &:~\Gamma^{\ifa}=\Gamma^{2} \coloneq \{\refpos \in \refraum: \refpos[2]=-1\}, \\
      \arr{\nicefrac{1}{\sqrt{3}},\nicefrac{1}{\sqrt{3}},\nicefrac{1}{\sqrt{3}}} &:~\Gamma^{\ifa}=\Gamma^{3} \coloneq \{\refpos \in \refraum: \refpos[1]+\refpos[2]+\refpos[3]=0\}, \\
      \arr{-1,0,0} &:~\Gamma^{\ifa}=\Gamma^{4} \coloneq \{\refpos \in \refraum: \refpos[1]=-1\}.
    \end{cases}
\end{align}
The corresponding determinant of the Jacobian on the surface $\Jref^{(\ifa)}(\refposquad_{ij}), \ i,j=0,...,\ppn,$ is defined as
\begin{align}
  \Jref^{(\ifa)}(\refposquad_{ij}) =
  \begin{cases}
    \frac{(1-\refposquad[2]_j)}{2}           &: \ifa = 1, \\
    \frac{(1-\refposquad[3]_j)}{2}           &: \ifa = \{2,4\}, \\
    \sqrt{3} \frac{(1-\refposquad[3]_j)}{2}  &: \text{otherwise},
  \end{cases}
\end{align}
with $\Jref^{(\ifa)} \in \R^{(\ppn+1) \times (\ppn+1)}$.

\subsubsection{Hexahedral and Quadrilateral Elements}
The reference domain of the unit hexahedral is $\refraumcol = \{\refpos \in [-1,1]^3\}$ with $M(\ppn) = (\ppn+1)^3$ orthonormal
basis functions
\begin{align}
  \testfunc_{ijk}(\refpos) = \mathcal{P}_i^{(0,0)}(\refposquad[1]) \mathcal{P}_j^{0,0)}(\refposquad[2])
  \mathcal{P}_k^{(0,0)}(\refposquad[3]), \ i,j,k = 0,...,\ppn.
\end{align}
The mapping $\mappingcol$ and its inverse reduce to the mapping $\mappingcol: \refpos \mapsto \refposquad$  and
${\mappingcol}^{-1}: \refposquad \mapsto \refpos$, respectively. As such, the adjoint and the inverse of the determinant of the corresponding
Jacobian are $\adj \Jmref=\mathbf{I}$ and $\Jref^{-1} = 1$ with the unit matrix $\mathbf{I}\in\R^{d\times d}$.
The normalized normal vectors $\normalvecref$ in the hexahedral reference space for each of the six faces are given as
\begin{align}
  \normalvecref^{(\ifa)} = \begin{cases}
      \arr{0,0,-1} &:~\Gamma^{\ifa}=\Gamma^{1} \coloneq \{\refpos \in \refraum: \refpos[3]=-1\}, \\
      \arr{0,-1,0} &:~\Gamma^{\ifa}=\Gamma^{2} \coloneq \{\refpos \in \refraum: \refpos[2]=-1\}, \\
      \arr{-1,0,0} &:~\Gamma^{\ifa}=\Gamma^{3} \coloneq \{\refpos \in \refraum: \refpos[1]=-1\}, \\
      \arr{ 1,0,0} &:~\Gamma^{\ifa}=\Gamma^{4} \coloneq \{\refpos \in \refraum: \refpos[1]=1\}, \\
      \arr{ 0,1,0} &:~\Gamma^{\ifa}=\Gamma^{5} \coloneq \{\refpos \in \refraum: \refpos[2]=1\}, \\
      \arr{ 0,0,1} &:~\Gamma^{\ifa}=\Gamma^{6} \coloneq \{\refpos \in \refraum: \refpos[3]=1\},
    \end{cases}
\end{align}
and the corresponding determinant of the Jacobian on the surface $\Jref^{(\ifa)} = \mathbf{I} \in  \R^{(\ppn+1) \times (\ppn+1)}$.
The 2D equivalent to a hexahedron is a quadrilateral element. As such, in 2D, the corresponding mapping reduces to the reference
quadrangle, given as $\refraumcol = \{\refpos \in
[-1,1]^2\}$ with $M(\ppn) = (\ppn+1)^2$, i.e., the third direction and its corresponding faces $\Gamma^1$ and $\Gamma^6$ are dropped.

\section{Nodal Discontinuous Galerkin Spectral Element Method on Mixed Elements}%
\label{sec:methods}

In this section, the discontinuous Galerkin spectral element method on hybrid meshes is introduced.
In the present work, Legendre--Gauss nodes are utilized as collocation points in preference to Legendre--Gauss--Lobatto and
Legendre--Gauss--Radau nodes due to the increased integration accuracy.
First, the weak form of the DGSEM is presented.
Then, the entropy-stable DGSEM based on the entropy variables and the generalized summation-by-parts operator is outlined, followed by a brief discussion of the treatment of viscous fluxes.
This section concludes with a discussion on the temporal integration of the resulting DGSE scheme and an analytical proof of the free-stream preservation and entropy conservation/stability properties.
\newline
The numerical methods presented in this paper are implemented in the open-source high-order solver
FLEXI\footnote{\url{https://github.com/flexi-framework/flexi}}~\cite{Krais2019}.
The meshes are generated with the high-order pre-processor PyHOPE\footnote{\url{https://pypi.org/project/PyHOPE/}}.

\subsection{Weak Form}

The computational domain $\Omega \subset \Rd$ with the number of dimensions $d \in \lbrace 2,3 \rbrace$ is composed of $\nel$ non-overlapping polytopal
curvilinear elements $\{\Omega^{(\iel)}\}_{\iel \in \{1:\nel\}}$.
Curvilinear elements are approximated by the corresponding orthonormal basis functions of degree $\ppngeo$ introduced in \cref{sec:theory}.
The boundary of each computational element $\partial \Omega^{(\iel)}$ is partitioned into faces $\{\Gamma^{\ifa}\}_{\ifa \in
\{1:\Nf\}}$, $\Gamma^{\ifa} \subset \partial \Omega^{(\iel)}$ with an outward unit normal vector $\normalvec$, and $\Nf$ denotes
  the number of element faces.
Each element is transformed to the polytopal reference space by the inverse of the previously introduced mapping $\mapping^{(\iel)} : \refraumcol \to \Omega^{(\iel)}, \refpos
\mapsto \mathbf{x}$ from the
reference element $\refpos \in \refraumcol$ to the physical space $\mathbf{x} \in \Omega$, cf.~\cref{fig:theory:mapping}.
The Jacobian matrix of this mapping is defined as $\Jm^{(\iel)} = \lr{\gradientXI \mapping^{(\iel)}}$ with its corresponding determinant $\J^{(\iel)}=\det\Jm^{(\iel)}$.
Subsequently, each polytopal reference element is mapped to the unit element $\refraum$ via the inverse of the mapping $\mappingcol$ and projected onto the $L_2$-space comprised of polynomials
$\testfunc \in \P(\refraum,\R^{\nvar})$ of order $\ppn$, which are chosen to be the tensor-product of 1D Lagrange polynomials
$\testfunc_{ijk}=\lagrange_i(\refposquad[1])\lagrange_j(\refposquad[2])\lagrange_k(\refposquad[3])$ with $i,j,k=0,\ldots,\ppn$.
The nodal solution in the reference element $\refraum$ is approximated by 1D Lagrange polynomials of degree $\ppn$, exploiting the Galerkin
property. This leads to
\begin{align}
  \consh^{(\iel)}(\mappingcol(\refposquad),t) = \mathbb{I}^{\ppn}(\cons) := \sum_{i,j,k=0}^\ppn \check{\cons}^{(\iel)}_{ijk}(t)
  \lagrange_i(\refposquad[1])\lagrange_j(\refposquad[2])\lagrange_k(\refposquad[3]) \in \R^{\nvar\times\nq}
\end{align}
with the nodal degrees of freedom $\check{\cons}_{ijk}^{(\iel)}(t) \in \R^{\nvar}$ in the $\iel$-th element, and the number of interpolation nodes in the volume $\nq \in \N$.
For simplicity, the subscript $h$ is omitted from $\cons$ in the following.
Finally, integration by parts leads to the weak form of \cref{eq:theory:euler}, given as
\begin{align}
  \projE{\diagJref \diagJ^{(\iel)} \cons_{t}^{(\iel)},\testfunc(\refposquad)} -& \projE{\fphysref^{(\iel)},
  \nabla_{\refposquad} \testfunc(\refposquad)} \nonumber \\
    +& \intpE{(\fphysref^{(\iel)}\cdot\normalvecref^{(\iel)})^\ast \testfunc(\refposquad)}  = \Null,
  \label{eq:dgsem_weak}
\end{align}
where $\projE{\cdot,\cdot}$ denotes the $L_2$ scalar product on $\refraum$, $\diagJref = \diag(\Jref)$, and $\diagJ^{(\iel)} = \diag(\J^{(\iel)})$.
Contravariant fluxes $\fphysref^{(\iel)}$ include the physical flux $\fphys^{(\iel)}$ and the metric terms $(\M^{(\iel)})^\transpose$, written as
\begin{align}
  \fphysref^{(\iel)} = (\adj \Jmref)^{(p)} (\M^{(\iel)})^\transpose \fphys^{(\iel)}, \
  (\M^{(\iel)})=\lr{\adj  \Jm^{(\iel)} }^\transpose \otimes \Iunit_{\nvar}.
  \label{eq:fcontra}
\end{align}
Since interpolation and differentiation only commute if the discretization error is zero, otherwise
$\mathbb{I}^{\ppn}(\mathbf{q}') \neq (\mathbb{I}^{\ppn}(\mathbf{q}))'$, cf.~\citep[22]{Kopriva2009}, these metric terms are approximated by the
conservative curl form derived in~\cite{Kopriva2006} to guarantee that the metric identities, $\divXI \adj \Jm = \mathbf{0}$, hold on the discrete level.
To summarize, the metric terms are computed as follows.
\begin{enumerate*}
  \item[a)] The metric terms, i.e., the conservative curl form, are computed on the corresponding polytopal reference element using the orthonormal basis functions introduced in \cref{sec:theory}.
  \item[b)] These metric terms are interpolated to the quadrature nodes on the polytopal reference element.
  \item[c)] The mapping of the metric terms to the unit cube reference element is achieved by means of the \textit{Duffy transformation}, in other words, by the corresponding adjoint of the Jacobian $(\adj \Jmref)$.
\end{enumerate*}
\newline
Neighboring elements are weakly coupled via the numerical flux normal to the element faces $\{\Gamma^{\ifa}\}_{\ifa \in
\{1:\Nf\}}$, resulting in
\begin{align}
  (\fphysref^{(\iel)}\cdot\normalvecref^{(\ifa)})^\ast
  =  \diagJ^{(\iel,\ifa)} (\fphys^{(\iel,\ifa)}\cdot\normalvec^{(\ifa)})^\ast
  =  \diagJ^{(\iel,\ifa)} f^\ast(\cons^{(\ifa),+} , \cons^{(\ifa),-}; \normalvec^{(\ifa)}).
\end{align}
The term $f^\ast$ denotes the numerical flux function, here approximated via Roe's numerical flux with the entropy fix
by~\citet{Harten1983b} or the local Lax-Friedrichs (Rusanov) flux, and $\Nf$ denotes the number of element surfaces.
The outward-pointing physical normal vector $\normalvec \in \Rd$ is computed via Nanson's formula
\begin{align}
  \abs{(\adj {\Jm^{(\iel)}})^\transpose \normalvecref^{(\ifa)}} \normalvec^{(\ifa)} = (\adj
    {\Jm^{(\iel)}})^\transpose \normalvecref^{(\ifa)}
    \label{eq:nanson}
\end{align}
with the surface element $\J^{(\iel,\ifa)} = \abs{(\adj {\Jm^{(\iel)}})^\transpose \normalvecref^{(\ifa)}}, \ \diagJ^{(\iel,\ifa)} = \diag(\J^{(\iel,\ifa)}),$ and the unit normal vector in the polytopal reference space $\normalvecref^{(\ifa)}$, see~\cref{sec:theory}.
The approximate solutions on the left $\cons^{(\ifa),+}$ and right $\cons^{(\ifa),-}$ of a face $\Gamma^{\ifa}$ are given as 1D
  volume operations due to the tensor-product nature of the interpolation using the surface Vandermonde
\begin{align}
  (\Vm_f^{(\ifa)})_{ij} = \lagrange_j(\refposquad^{(\ifa)}_i), \Vm_f \in \R^{\nf\times\nq},
\end{align}
where $\refposquad^{(\ifa)}$ are the $\nf$ quadrature nodes on the respective facets, see~\cref{sec:theory}.
This results in
\begin{align}
    {\cons}^{(\ifa)}(\mappingcol(\refposquad),t) = \Vm_f^{(\ifa)} {\cons}^{(\iel)}
      (\mappingcol(\refposquad),t).
\end{align}

The discrete form of~\cref{eq:dgsem_weak} is obtained using Legendre--Gauss quadrature on $\nq$ Legendre--Gauss nodes, where the interpolation and integration nodes are collocated, yielding an integration accuracy of $2\ppn+1$ in 1D.
This requires the definition of the following operators as
\begin{align}
  \Wm &= \diag(\vvec(\wbold \otimes \wbold \otimes \wbold)), & \{\w\}_{i=0}^\ppn,\nonumber \\
  \Dm_{ij} &= \lagrange_{j}' ({\refposquadq}_i), & i,j = 1,...,\nq, \\
  \Wm_f &= \diag(\vvec(\wbold \otimes \wbold)),\nonumber
\end{align}
which are the mass matrix $\Wm \in \R^{\nq\times\nq}$ and the polynomial differentiation operator $\Dm \in \R^{\nq\times\nq}$ composed of exact derivatives of the Lagrange polynomials.
The surface weights are denoted as $\Wm_f \in \R^{\nf\times\nf}$. %
The $\vvec(\cdot)$ operator flattens a matrix into a one-dimensional vector, and $\otimes$ indicates the dyadic product.
The semi-discrete weak form is then given as
\begin{align}
  \Wm \diagJref(\refposquadq) \diagJ^{(\iel)}(\mappingcol(\refposquadq)) \cons_t^{(\iel)} & \nonumber \\
  - \sum_{p=1}^{d} \Wm \Dm^\transpose \left({(\adj \Jmref)^{(p)} {\M^{(\iel)}}^\transpose}\right) \fphys^{(\iel,p)}(\cons^{(\iel)}(\mappingcol(\refposquadq)))) & \nonumber \\
  + \sum_{\ifa=1}^{\Nf} (\Vm_f^{(\ifa)})^\transpose \Wm_f \diagJref^{(\ifa)} \diagJ^{(\iel,\ifa)} f^\ast(\cons^{(\ifa),+} , \cons^{(\ifa),-}; \normalvec^{(\ifa)}) & = \ \Null,
  \label{eq:dgsem_weak_discrete}
\end{align}
with the determinant of the Jacobian on the surface, denoted as $\diagJref^{(\ifa)}$, see~\cref{sec:theory} for its definition.

\subsection{Entropy-Stable DGSEM on Legendre--Gauss Nodes}

In this work, an entropy-stable DGSEM on Legendre--Gauss nodes on hybrid meshes is introduced using entropy variables and the generalized decoupled/hybridized SBP operator
proposed by~\citet{Chan2018}
\begin{align}
  \Qm = \left(\begin{matrix}
    \Wm \Dm - \nicefrac{1}{2} \Vm_f^\transpose \Wm_f \Vm_f &~ \nicefrac{1}{2} \Vm_f^\transpose \Wm_f \\
    -\nicefrac{1}{2} \Wm_f \Vm_f                           &~ \nicefrac{1}{2} \Wm_f
  \end{matrix}\right)
  \label{eq:sbp}
\end{align}
with the property $\Qm^\transpose + \Qm = \mathrm{diag}(\Null,\Wm_f)$.
To the author's knowledge, this is the first time an entropy-stable DGSEM has been applied to mixed element meshes including hexahedral, prismatic, tetrahedral, and pyramidal elements.
Following~\cite{Chan2018,Chan2019}, the matrix formulation of~\cref{eq:dgsem_weak_discrete} for the entropy-stable DGSEM on mixed element meshes can be retrieved as
\begin{align}
  \Wm \diagJref(\refposquadq) \diagJ^{(\iel)}(\mappingcol(\refposquadq)) \cons_t^{(\iel)}
  + \sum_{p=1}^{d} \left[\begin{matrix} \mathbf{I} & \Vm_f^{\transpose} \end{matrix} \right] \lr{2 \mathbf{Q} \circ
  \fphysref^{(p)}(\qentropyvar^{(\iel)})} \mathbf{1} & \nonumber \\
  + \sum_{\ifa=1}^{\Nf} (\Vm_f^{(\ifa)})^\transpose \Wm_f
  \lr{\diagJref^{(\ifa)} \diagJ^{(\iel,\ifa)} f^\ast(\qentropyproj^+,\qentropyproj^-;\normalvec^{(\ifa)}) - (\fphysref(\qentropyproj^{(\ifa)}) \cdot \normalvecref^{(\ifa)})} &= \Null,
    \label{eq:esdgsem_matrix}
\end{align}
with the vector of all ones $\mathbf{1} \in \R^2$, $\qentropyvar^{(\iel)} = \entropyvar(\cons^{(\iel)}(\mappingcol(\refposquadq)))$, and $\qentropyproj$ are the entropy-projected variables at the facets
\begin{align}
  \qentropyproj = \cons(\entropyvarsurf), \hspace{1cm}
  \entropyvarsurf = \Vm_f \entropyvarvol(\cons).
\end{align}
The symbol $\fphysref^{(p)} \in \R^{\nvar \times (\nq+\nf)}$ denotes the contravariant entropy conservative (and/or kinetic energy preserving) fluxes in the volume and on the surface based on an adequate two-point flux $\fsharp$ according to~\citet{Tadmor1987} or~\citet{Jameson2008}.
The contravariant entropy conservative two-point fluxes on the surface are indicated by $\fphysref(\qentropyproj^{(\ifa)})$.
Using~\cref{eq:fcontra}, the term $\fphysref(\qentropyproj^{(\ifa)}) \cdot \normalvecref^{(\ifa)}$ can be also written as $\fphysref(\qentropyproj^{(\ifa)}) \cdot \normalvecref^{(\ifa)} = \diagJref^{(\ifa)} \diagJ^{(\iel,\ifa)} (\fphys(\qentropyproj^{(\ifa)}) \cdot \normalvecref^{(\ifa)})$.
The numerical flux at the cell boundaries, given as
\begin{align}
  f^\ast(\qentropyproj^+,\qentropyproj^-;\normalvec^{(\ifa)}) = \fsharp ({\qentropyproj}^+,{\qentropyproj}^-) + \frac{1}{2} \mathcal{H}
    \jump{\qentropyproj}
\end{align}
includes the entropy conservative numerical flux function and a matrix dissipation operator.
Here, a Roe-type penalization, denoted as $\mathcal{H}$, is used, which has been proven to avoid the order reduction generally observed for entropy stable methods~\cite{Hindenlang2019}.
The jump operator $\jump{\cdot}$ is defined as $\jump{\cdot} = [(\cdot)^+-(\cdot)^-]$.

With \cref{eq:sbp}, \cref{eq:esdgsem_matrix} can also be written as
\begin{align}
  \Wm \diagJref(\refposquadq) \diagJ^{(\iel)}(\mappingcol(\refposquadq)) \cons^{(\iel)}_t = \mathcal{R}^{(\iel)}, \quad
  \mathcal{R}^{(\iel)} \coloneq - \mathcal{Q}^{(\iel)} - \mathcal{E}^{(\iel)}.
  \label{eq:esdgsem_discrete}
\end{align}
For this, the following operators are defined
\begin{align}
  \mathcal{E}^{(\iel)} =& \sum_{\ifa=1}^{\Nf} (\Vm_f^{(\ifa)})^\transpose \Wm_f \diagJref^{(\ifa)} \left[\mathcal{R}^{(\iel,\ifa)} + \diagJ^{(\iel,\ifa)}
    f^\ast(\qentropyproj^{(\ifa),+} , \qentropyproj^{(\ifa),-}; \normalvec^{(\ifa)}) \right], \\
  \mathcal{Q}^{(\iel)} =& \ \sum_{p=1}^{d} \left[(\adj \Jmref)^{(p)} \ \avg{(\M^{(\iel)})^\transpose}
   ~\fsharp^{(p)}(\qentropyvar^{(\iel)}) \ \Wm \Dm \right. \nonumber \\
                & - \left. \Dm^\transpose \Wm \ (\adj \Jmref)^{(p)} \ \avg{(\M^{(\iel)})^\transpose}
   ~\fsharp^{(p)}(\qentropyvar^{(\iel)}) \right]
\end{align}
and
\begin{align}
  \mathcal{R}^{(\iel,\ifa)} =
  -         &  {\Vm_f^{(\ifa)}} \fsharp(\qentropyvar^{(\iel)}) \left(\avg{(\M^{(\iel)})^\transpose} \cdot \normalvecref^{(\ifa)}\right) \nonumber \\
  -         &  {\Vm_f^{(\ifa)}} \fsharp(\qentropyvar^{(\iel)},\qentropyproj^{(\ifa)}) \left(\frac{1}{2}\lr{(\M^{(\iel)})^\transpose +
               \adj \Jm^{(\iel,\ifa)}} \cdot \normalvecref^{(\ifa)}\right) \nonumber \\
  +         &  {\color{white}{\Vm_f^{(\ifa)}}} \fsharp(\qentropyvar^{(\iel)},\qentropyproj^{(\ifa)}) \left(\frac{1}{2}\lr{({\M^{(\iel)}})^\transpose +
               \adj \Jm^{(\iel,\ifa)}} \cdot \normalvecref^{(\ifa)}\right),
\end{align}
where the definition of the contravariant fluxes in the volume, given as
\begin{align}
  (\fphysref_q^{(p)})_{ij} &= (\adj \Jmref)^{(p)} \avg{(\M^{(\iel)})^\transpose}_{ij}~\fsharp^{(p)}
  ({\entropyvar}_i^{(\iel)},{\entropyvar}_j^{(\iel)}),
\end{align}
is utilized.
Here, $\avg{\cdot}_{ij}= 0.5 [(\cdot)_i+(\cdot)_j]$ denotes the averaging operator.
The Jacobian on the surface is defined as $\Jm^{(\iel,\ifa)} = \Vm_f \Jm^{(\iel)}$.

\subsubsection{Flux Computation on Triangular Facets}
It is important to note that the numerical flux function employed in~\cref{eq:dgsem_weak_discrete} assumes that the quadrature nodes
on the faces are unique, i.e., their locations on the left and the right face coincide. This is trivially fulfilled for quadrilateral and triangular faces in the polytopal reference space, but not
necessarily true for collapsed triangular faces which can be arbitrarily collapsed depending on the orientation of element
local coordinate system.
In order to surmount this limitation, the authors in~\cite{Warburton1995} proposed an algorithm which orients the local coordinate system in
each element to ensure that the quadrature nodes on the triangular facets are unique.
In this work, a different approach is adopted to ensure unique nodes on triangular facets for arbitrarily unstructured
grids.
For this, the extrapolated solutions on the left and right of an element interface are transformed
from the collapsed coordinate system to the polytopal reference domain $\refraumcol$ with $\nf(\ppn)$ quadrature nodes and weights
on the triangular facet given by~\citet{Xiao2010}. The reader is referred to~\cite{Xiao2010} for further details on the
corresponding quadrature rule.
The numerical flux function is evaluated on these quadrature nodes and the flux is mapped back from $\refraumcol$ to the
collapsed coordinate system of each of the adjacent volumes.
The procedure can be summarized as follows: %
\begin{enumerate*}[label={\alph*)}]
  \item The entropy-projected solution left and right to a triangular element interface between a tetrahedron and another tetrahedron, a prism, or
a pyramid is projected from the $\nf=(\ppn+1)^2$ quadrature points on the face to the quadrature nodes on the triangle given by~\cite{Xiao2010}. This ensures that the solution on the left and right of this interface is defined on unique integration points.
  \item The numerical flux computation is performed on this point set.
  \item The final flux is projected back to the local coordinate system of the corresponding adjacent elements.
\end{enumerate*}

\subsubsection{Second-Order Terms}
Evaluation of the viscous fluxes in the NSE requires the computation of the gradients of the primitive variables, $\nabla \cons^{\text{prim}}$.
In this work, the BR1 (lifting) scheme~\cite{Bassi1997} is utilized to approximate the gradients of $\cons^{\text{prim}}$, %
which requires satisfying an additional set of equations, given as
\begin{align}
  \mathbf{g} = \gradientX \cons^{\text{prim}},
\end{align}
where $\mathbf{g}$ are the lifted gradients.
The solution to this supplementary set of equations is obtained by deriving its weak form and applying
the DGSEM, similarly to the procedure above (see, e.g.,~\cite{Krais2021}), leading to
\begin{align}
  \Wm \diagJref(\refposquad) \diagJ^{(\iel)}(\refpos) \mathbf{g} =& -\sum\limits_{p=1}^{d}\Wm \Dm^\transpose
  (\adj \Jmref)^{(p)} \ \avg{(\M^{(\iel)})^\transpose} \cons^{\text{prim}} \nonumber \\
                                                          &+
                                                        \sum_{\ifa=1}^{\Nf} ({\Vm_f^{(\ifa)}})^\transpose \Wm_f \diagJref^{(\ifa)}
                                                          \diagJ^{(\iel,\ifa)}\avg{\cons^{\text{prim},(\ifa)} \cdot
                                                            \normalvec^{(\ifa)}}.
  \label{eq:lifting}
\end{align}
The final equation can be expressed as
\begin{align}
  \Wm &\diagJref(\refposquad) \diagJ^{(\iel)}(\refpos) \cons^{(\iel)}_t = \mathcal{R}^{(\iel)}, \nonumber \\
  \mathcal{R}^{(\iel)} \coloneq - \mathcal{Q}^{(\iel)} - \mathcal{E}^{(\iel)} & - \sum\limits_{p=1}^{d}\Wm
  \Dm^\transpose (\adj \Jmref)^{(p)} \ \avg{(\M^{(\iel)})^\transpose}\fphys_v^{(p)}(\cons, \mathbf{g}) \nonumber \\
                                                                              & + \sum_{\ifa=1}^{\Nf}
                                                                                ({\Vm_f^{(\ifa)}})^\transpose \Wm_f \diagJref^{(\ifa)}
                                                                                \diagJ^{(\iel,\ifa)}\avg{\fphys_v^{(\ifa)} \cdot \normalvec^{(\ifa)}}.
  \label{eq:dgsem_lifting}
\end{align}
with $\avg{{\fphys_v}^{(\ifa)} \cdot \normalvec^{(\ifa)}}$ as the average viscous flux at an element interface.

\subsubsection{Modal Time-Stepping}

A straight-forward temporal integration of the tensor-product approximation in collapsed coordinates can result in excessively small
time steps~\cite{Dubiner1991}.
According to~\cite{Montoya2024b}, this can be avoided by advancing the modal polynomial coefficients in time instead of the nodal degrees of freedom.
For this, the \enquote{warped} tensor-product structure of the \PKD~orthogonal modal polynomial basis in conjunction with a weight-adjusted approximation of the inverse of the curvilinear mass matrix is exploited. %
Their combination yields a computationally and memory efficient scheme~\cite{Montoya2024a}.
This scheme may be interpreted as an over-integration approach, in which case the modal-based formulation is advanced in time, but a nodal-based
integration with more quadrature points than its modal counterpart is utilized, similar to~\cite{Gassner2009}.
It is evident that this offers the potential for a reduction in computational cost by exploiting tensor-product properties, while simultaneously avoiding excessively small time steps.
The modal approximation is given as
\begin{align}
  \consh^{(\iel)}(\refpos,t) = \sum_{i=1}^{M(\ppn)} \tilde{\cons}^{(\iel)}_{i}(t)
  \testfunc_i(\refpos) \in \R^{\nvar\times M(\ppn)}
\end{align}
with the modal expansion coefficients $\tilde{\cons}^{(\iel)}(t)$ in the $\iel$-th element and the orthonormal basis
$\{\testfunc_i\}_{i\in\{1:M(\ppn)\}}$, defined in~\cref{sec:theory}.
The mapping from modal to nodal space can be written using the generalized Vandermonde matrix $\Vm \in \R^{\nq \times M(\ppn)}$ as
\begin{align}
  \cons^{(\iel)}(t) = \Vm \tilde{\cons}^{(\iel)}(t), \quad \Vm_{ij} = \testfunc_j(\refpos_i).
\end{align}
The modal representation is obtained by solving the following projection problem
\begin{align}
  \Mm^{(\iel)} \tilde{\cons}^{(\iel)} = \Vm^\transpose \Wm \diagJref \diagJ^{(\iel)} \cons^{(\iel)}
\end{align}
with the physical mass matrix $\Mm^{(\iel)} = \Vm^\transpose \Wm \diagJModal^{(\iel)} \Vm$.
Following~\cite{Chan2019a}, the diagonal matrix of the determinant of the Jacobian $\diagJ^{(\iel)}$ from the computational domain to the polytopal
reference space is approximated by the projection of $\diagJ^{(\iel)}$ onto $\P(\refraumcol)$ to ensure discrete
conservation, resulting in $\diagJModal^{(\iel)} = (\widetilde{\Mm})^{-1} \Vm^\transpose \Wm \diagJ^{(\iel)}$ with $\widetilde{\Mm} = \Vm^\transpose \Wm \Vm$.
The mass matrix $\Mm^{(\iel)}$ is dense for non-affine mappings from the reference to the physical element, leading to elevated memory
requirements and an augmented computational effort. This can be avoided through the use of the weight-adjusted approximation of the mass matrix inverse proposed
by~\cite{Chan2019a}, written as
\begin{align}
  (\Mm^{(\iel)})^{-1} \approx (\widetilde{\Mm})^{-1} \Vm^\transpose \Wm (\diagJModal^{(\iel)})^{-1} \Vm (\widetilde{\Mm})^{-1}.
\end{align}
Finally, the time derivative of the weight-adjusted modal formulation is
\begin{align}
  \tilde{\cons}^{(\iel)}_t = (\widetilde{\Mm})^{-1} \Vm^\transpose \Wm (\diagJModal^{(\iel)})^{-1} \Vm
  (\widetilde{\Mm})^{-1} \Vm^\transpose \mathcal{R}^{(\iel)}(t).
  \label{eq:dgsem_modal}
\end{align}
The matrix $\widetilde{\Mm}$ reduces to a unit matrix if an orthonormal basis function $\testfunc$ and an integration rule with at least order $2\ppn$ is chosen.

\subsection{Free-stream Preservation, Discrete Entropy Conservation and Stability}
In this section, we prove that the entropy-stable \DGSEM~in \cref{eq:dgsem_modal} with the residuum as defined in \cref{eq:dgsem_lifting} is free-stream preserving and discretely entropy conservative/stable.
It is important to note that entropy conservation/stability is demonstrated exclusively for the hyperbolic part.
The entropy stability of the viscous operator has been proven in the study by~\cite{Dalcin2019}.
The proof for entropy conservation/stability of the non-unique quadrature nodes on triangular facets can be derived by directly following  \cite{Chan2021} and the procedure outlined below.
A numerical demonstration is provided in~\cref{sec:validation}.
\begin{theorem}
  Assume a domain with periodic boundary conditions and a constant solution in each element $\iel=1,...,\nel$, given as $\cons^{(\iel)} = \mathbf{C} \in \R^{\nq}$.
  Let the numerical fluxes $f^\ast$ and two-point fluxes $\fsharp$ be consistent.
  Suppose that the metric identities are satisfied, $\divXI \Jm = \nabla_\eta \cdot \lr{\Jmref^{-1} \Jm} = \Null$, and the normal vectors are computed using~\cref{eq:nanson}, such that the decoupled SBP operator in~\cref{eq:sbp} satisfies the discrete metric identities $\Qm ~ \mathbf{1} = \mathbf{0}$.
  Let the extrapolation operator $\Vm_f$ be exact for a constant state and the quadrature rule is at least of degree $2\ppn$ to
  ensure the exactness of the volume quadrature.
  Then, the semi-discrete \DGSEM~in \cref{eq:dgsem_modal} with the residuum $\mathcal{R}^{(\iel)}(t)$ as defined in \cref{eq:dgsem_lifting} and \cref{eq:lifting} is free-steam preserving, i.e., discretely conservative, satisfying
  \begin{align}
    \Wm \diagJref \diagJ^{(\iel)} \cons^{(\iel)}_t =& - \sum_{\ifa=1}^{\Nf} (\Vm_f^{(\ifa)})^\transpose \Wm_f \diagJref^{(\ifa)} \diagJ^{(\iel,\ifa)} f^\ast(\qentropyproj^{(\ifa),+} , \qentropyproj^{(\ifa),-}; \normalvec^{(\ifa)}).
  \end{align}
  \label{lemma1}
\end{theorem}
\begin{proof}
  Exploiting that the discrete metric identities, $\Dm \cdot \lr{\Jmref^{-1} \Jm} = \Null$, are satisfied and invoking the conservation property of the numerical interface fluxes, the lifted gradients reduce to $\mathbf{g} = \Null$.
  Therefore, \cref{eq:dgsem_lifting} simplifies to \cref{eq:esdgsem_discrete} which is equivalent to \cref{eq:esdgsem_matrix}.
  With the property $\Qm ~ \mathbf{1} = \mathbf{0}$, the discrete metric identities, and Gauss' theorem, \cref{eq:esdgsem_matrix} can be written as
  \begin{align*}
    \Mm^{(\iel)} \tilde{\cons}^{(\iel)}_t = -& (\mathbf{1}^{(\nq)})^\transpose \sum_{\ifa=1}^{\Nf} (\Vm_f^{(\ifa)})^\transpose \Wm_f \diagJref^{(\ifa)} \diagJ^{(\iel,\ifa)} f^\ast(\qentropyproj^{(\ifa),+} , \qentropyproj^{(\ifa),-}; \normalvec^{(\ifa)}).
  \end{align*}
  with $\Vm~\mathbf{1}^{(M(\ppn))} = \mathbf{1}^{(\nq)}$ for a constant state, and $\mathbf{1}^{(M(\ppn))} \in \R^{(M(\ppn))}$ is a vector composed of the modal expansion coefficients.
  This constitutes a statement of local conservation in regard to the weight-adjusted mass matrix.
  From \citep[Lemma 2]{Chan2019a}, the exactness of the volume quadrature yields
  \begin{align*}
    \Mm^{(\iel)} \tilde{\cons}^{(\iel)}_t = (\mathbf{1}^{(\nq)})^\transpose \Wm \diagJref \diagJ^{(\iel)} \cons^{(\iel)}_t
  \end{align*}
  Finally, global conservation holds on the discrete level by invoking the conservation property of the numerical interface fluxes.
  This completes the proof of \Cref{lemma1}.
\end{proof}

\begin{theorem}
  Let the numerical fluxes $f^\ast$ be consistent, conservative and entropy-stable, given as
  \begin{align*}
    \lr{\qentropyproj^{(\ifa),+} - \qentropyproj^{(\ifa),-}}^\transpose f^\ast(\qentropyproj^{(\ifa),+} , \qentropyproj^{(\ifa),-}; \normalvec^{(\ifa)}) \leq \lr{\boldsymbol{\Psi}(\qentropyproj^{(\ifa),+}) - \boldsymbol{\Psi}(\qentropyproj^{(\ifa),-})} \cdot \normalvec^{(\ifa)}
  \end{align*}
  with the entropy flux potential $\boldsymbol{\Psi} = \entropyvar^\transpose \fsharp - \fentropyref$.
  Suppose the two-point fluxes $\fsharp$ are consistent, symmetric and entropy conservative according to Tadmor's condition~\cite{Tadmor1987}.
  Assume that the metric identities are satisfied on the discrete level and that the normal vectors are computed using~\cref{eq:nanson}.
  Then, the semi-discrete \DGSEM~in \cref{eq:dgsem_modal} with the spatial operator defined in \cref{eq:esdgsem_matrix} is discretely entropy-stable, satisfying the following entropy balance
  \begin{align}
    \mathbf{1}^\transpose \Wm \diagJref \diagJ \mathbf{\entropy}^{(\iel)}_t \leq \sum_{\ifa=1}^{\Nf} \mathbf{1}^\transpose \Wm_f \diagJref^{(\ifa)} \diagJ^{(\iel,\ifa)} \lr{ \boldsymbol{\Psi} (\qentropyproj^{(\ifa)}) - f^\ast(\qentropyproj^+,\qentropyproj^-;\normalvec^{(\ifa)})}
  \end{align}
  with the vector of the entropy $\mathbf{\entropy}^{(\iel)} \in \R^\nq$.
  This entropy balance reduces to an equality for an entropy conservative numerical flux function.
  \label{lemma2}
\end{theorem}

\begin{proof}
  Focusing on the hyperbolic operator, entropy conservation/stability of the final \DGSE~operator can be analytically proven by following the approach described in~\cite{Chan2018,Chan2019a}, wherein \cref{eq:dgsem_modal} is first multiplied by the weight-adjusted approximation of the mass matrix $\Mm^{(\iel)}$ which results in the weight-adjusted approximation of \cref{eq:esdgsem_matrix}, given as
  \begin{align}
    \label{eq:esdgsem_weigthadjusted}
    \Mm^{(\iel)} \tilde{\cons}^{(\iel)}_t
    + \sum_{p=1}^{d} \left[\begin{matrix} \Vm^\transpose & \widetilde{\Vm}_f^{\transpose} \end{matrix} \right] \lr{2 \mathbf{Q} \circ
    \fphysref^{(p)}(\qentropyvar^{(\iel)})} \mathbf{1} & \\
    + \sum_{\ifa=1}^{\Nf} (\widetilde{\Vm}_f^{(\ifa)})^\transpose \Wm_f
    \lr{\diagJref^{(\ifa)} \diagJ^{(\iel,\ifa)} f^\ast(\qentropyproj^+,\qentropyproj^-;\normalvec^{(\ifa)}) -
  (\fphysref(\qentropyproj^{(\ifa)}) \cdot \normalvecref^{(\ifa)})} &= \Null. \nonumber.
  \end{align}
  Here, $\widetilde{\Vm}_f^\transpose = \Vm^\transpose \Vm_f^\transpose \in \R^{M(\ppn) \times \nf}$ denotes the transpose of the modal interpolation matrix to the surface quadrature nodes.
  Secondly, \cref{eq:esdgsem_weigthadjusted} is multiplied by the vector of the weight-adjusted projection of the entropy variables $\widetilde{\entropyvar}^{(\iel)} = (\Mm^{(\iel)})^{-1} \Vm^\transpose \Wm \diagJref \diagJ^{(\iel)} \entropyvar^{(\iel)}$.
  Using similar arguments as in \cite{Chan2019a}, this yields the following expression for the first term on the left-hand side
  \begin{align}
    (\widetilde{\entropyvar}^{(\iel)})^\transpose \Mm^{(\iel)} \tilde{\cons}^{(\iel)}_t = \mathbf{1}^\transpose \Wm \diagJref \diagJ \mathbf{\entropy}^{(\iel)}_t.
  \end{align}
  The second term can be rewritten as
  \begin{align}
    \label{eq:esdgsemvolume}
    &- (\widetilde{\entropyvar}^{(\iel)})^\transpose \sum_{p=1}^{d} \left[\begin{matrix} \Vm^\transpose ~ \widetilde{\Vm}_f^{\transpose} \end{matrix} \right] \lr{2 \mathbf{Q} \circ \fphysref^{(p)}(\qentropyvar^{(\iel)})} \mathbf{1} \\
    = &- \sum_{p=1}^{d} (\check{\entropyvar}^{(\iel)})^\transpose \lr{ (( \adj \Jmref)^{(p)})^\transpose \Wm_f \circ (\avg{(\M^{(\iel)})^\transpose} \fsharp^{(p)}(\qentropyvar^{(\iel)}))} \mathbf{1} \nonumber \\
    &- \sum_{p=1}^{d} (\check{\entropyvar}^{(\iel)})^\transpose \lr{ (( \adj \Jmref)^{(p)})^\transpose \lr{ \Qm - \Qm^\transpose } \circ (\avg{(\M^{(\iel)})^\transpose} \fsharp^{(p)}(\qentropyvar^{(\iel)}))} \mathbf{1}, \nonumber
  \end{align}
  where the definitions of the \SBP~operator, cf.~\cref{eq:sbp}, and the contravariant fluxes, see \cref{eq:fcontra}, are exploited.
  Here, $\check{\entropyvar}^{(\iel)} = \left[\begin{matrix} \Vm^\transpose ~ \widetilde{\Vm}_f^{\transpose} \end{matrix} \right] \widetilde{\entropyvar}^{(\iel)}$.
  Following \cite{Chan2019a} and using the metric identities, the second term in \cref{eq:esdgsemvolume} reduces to
  \begin{align}
    &- \sum_{p=1}^{d} (\check{\entropyvar}^{(\iel)})^\transpose \lr{ (( \adj \Jmref)^{(p)})^\transpose \lr{ \Qm - \Qm^\transpose } \circ (\avg{(\M^{(\iel)})^\transpose} \fsharp^{(p)}(\qentropyvar^{(\iel)}))} \mathbf{1} \nonumber \\
    = & - \sum_{\ifa=1}^{\Nf} \mathbf{1}^\transpose \Wm_f \diagJref^{(\ifa)} \diagJ^{(\iel,\ifa)} \boldsymbol{\Psi} (\qentropyproj^{(\ifa)}).
  \end{align}
  Similar to \cite{Chan2019a}, summation of the third term on the left-hand side of \cref{eq:esdgsem_weigthadjusted} multiplied by the projected entropy variables and the first term on the right-hand side of \cref{eq:esdgsemvolume} yields
  \begin{align}
    & \sum_{\ifa=1}^{\Nf} (\check{\entropyvar}_f^{(\ifa)})^\transpose \Wm_f \lr{\diagJref^{(\ifa)} \diagJ^{(\iel,\ifa)} f^\ast(\qentropyproj^+,\qentropyproj^-;\normalvec^{(\ifa)}) - (\fphysref(\qentropyproj^{(\ifa)}) \cdot \normalvecref^{(\ifa)})} \nonumber \\
    &- \sum_{p=1}^{d} (\check{\entropyvar}^{(\iel)})^\transpose \lr{ (( \adj \Jmref)^{(p)})^\transpose \Wm_f \circ (\avg{(\M^{(\iel)})^\transpose} \fsharp^{(p)}(\qentropyvar^{(\iel)}))} \mathbf{1} \nonumber \\
    =& \sum_{\ifa=1}^{\Nf} (\check{\entropyvar}_f^{(\ifa)})^\transpose \Wm_f \diagJref^{(\ifa)} \diagJ^{(\iel,\ifa)} f^\ast(\qentropyproj^+,\qentropyproj^-;\normalvec^{(\ifa)}).
  \end{align}
  This completes the proof of \Cref{lemma2}.
\end{proof}

\section{Validation and Application}%
\label{sec:validation}

In the following, the implementation of the proposed entropy-stable DGSE scheme is validated based on numerical experiments, starting with the transportation of a constant solution before moving on to a generic manufactured solution on curvilinear grids and the weakly compressible Taylor--Green vortex.
This allows validation of the free-stream preservation, convergence rates, and entropy conservation/stability of the DGSEM.
It is important to acknowledge that the objective of this paper is not to provide a comparison of DGSEM with other related DG schemes, such as multidimensional DG.
Such a comparison of the efficiency and accuracy of the tensor-product DGSEM and the multidimensional DG operator is extensively given in~\cite{Montoya2024b} for triangular and tetrahedral elements.
This section concludes with the simulation of the flow around the common research model to demonstrate the applicability of the proposed scheme to real-world problems.
It has to be noted that all quantities are non-dimensional.

\subsection{Free-Stream Preservation}

The free-stream preservation property, i.e., conservation of a constant state over time, is evaluated for an arbitrary constant state.
Here, the state $\cons$ is chosen with $\rho=1$, $\vel=0.3$ and $p=17.857$, which is advanced for ten time steps on a computational domain $\Omega \in [0,1]^3$ discretized by \num{248} polytopal elements.
The mesh is constructed by starting with a $8^3$ hexahedral elements grid, where each hexahedral element is either kept as one
hexahedron (\texttt{HEXA}) or split into two prisms (\texttt{PRIS}), six pyramids (\texttt{PYRA}), or six tetrahedrons
(\texttt{TETRA}), as illustrated in~\cref{fig:split_hexa}.
To circumvent the generation of mixed element interfaces, a pyramid is, if necessary, divided into two tetrahedral elements.
This results for the given configuration in a total of four zones where each zone consists of \num{16} hexahedrons, \num{32} prisms,
\num{88} pyramids, or \num{112} tetrahedrons, respectively.
The ensuing numerical test cases will employ the same procedure for mesh generation.
\begin{figure}
  \includegraphics[width=\linewidth]{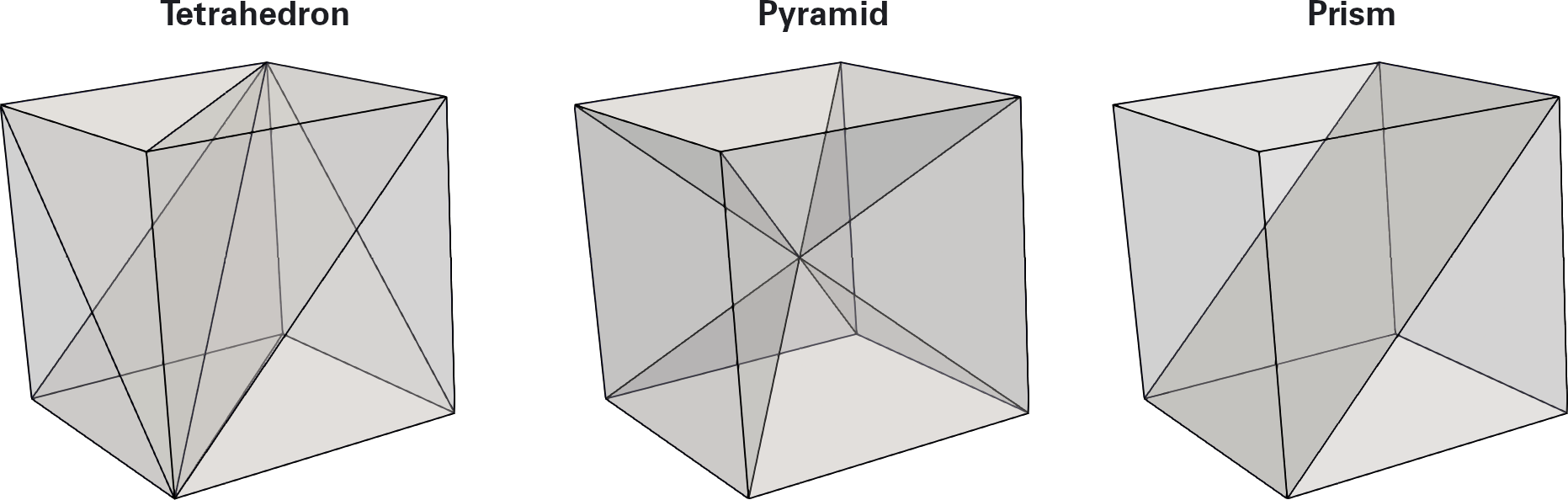}
  \caption{Exemplary split of one hexahedral element into two prisms, six pyramids, or six tetrahedrons (from right to left).}
  \label{fig:split_hexa}
\end{figure}
The results in~\cref{tab:validation:freestream} highlight the free-stream preservation property of the DGSEM.
Similar results are obtained on a two-dimensional computational domain, where the error is in the order of machine precision.
\begin{table}[htb!]
  \centering
  \sisetup{output-exponent-marker=\ensuremath{\mathrm{E}},round-mode=places,round-precision=5}
  \resizebox{\textwidth}{!}{\begin{tabular}{c|ccccc}
                & $\rho$                & $\rho \vel[1]$        & $\rho \vel[2]$        & $\rho \vel[3]$        & $\rho e$              \\ \hline
      $\ppn=3$  & \num{1.969789014E-15} & \num{1.226776505E-13} & \num{2.257047358E-13} & \num{1.610594680E-13} & \num{4.647459970E-13} \\
      $\ppn=4$  & \num{4.305873433E-15} & \num{9.175273378E-14} & \num{1.154184587E-13} & \num{1.310743385E-13} & \num{3.002585069E-13} \\
      $\ppn=5$  & \num{8.894536355E-15} & \num{9.569395831E-14} & \num{1.259617535E-13} & \num{9.444142757E-14} & \num{2.902739569E-13} \\\hline
  \end{tabular}}
  \caption{Free-stream preservation: $L_2$-error of $\cons$.}%
  \label{tab:validation:freestream}
\end{table}

\subsection{Experimental Order of Convergence}
In the following, the $h$- and $p$-convergence properties of the entropy-stable DGSEM on curvilinear meshes are investigated on a
computational domain $\Omega \in [0,1]^d$ with $d=2,3$ using the method of manufactured solutions.
Following~\cite{Hindenlang2012}, the exact solution is assumed to be of the form
\begin{align}
  \rho = 2+0.1\sin(2\pi(\pos[1]+\pos[2]+\pos[3]-0.3t)), \ \rho \vel = \rho, \ \rho e = (\rho)^2.
\end{align}
In the $p$-convergence survey, the grid size is fixed to $4^3$ hexahedral elements in each direction and the polynomial degree $\ppn$ is varied.
For the $h$-convergence study, $\Omega$ is discretized using $1^3$ to $16^3$ hexahedral elements in each direction and $\ppn=4$.
For both investigations, similar to above, each hexahedral element is either kept as one hexahedron (\texttt{HEXA}) or split into two prisms (\texttt{PRIS}), six pyramids (\texttt{PYRA}), or six tetrahedrons (\texttt{TETRA}).
In two dimensions, each quadrilateral element is kept as one quadrilateral (\texttt{QUAD}) or divided into two triangular elements (\texttt{TRIA}).
In addition, a mixed elements grid (\texttt{MIX}) composed of all four element types is considered, resulting in a total of \num{32}
up to \num{15488} polytopal elements during the $h$-convergence study, while the $p$-convergence study is conducted on a \num{248}
element mesh, cf.~\cref{fig:convtest_tgv} (left).
Following~\cite{Chan2019}, the mesh is sinusoidally deformed by perturbing the node positions as
\begin{align*}
  \mathrm{2D}:& \ \pos = \left[ \begin{matrix}
    \pos[1] + \epsilon L
    \cos\lr{\nicefrac{\pi}{L}(\pos[1]-\nicefrac{1}{2})}\cos\lr{\nicefrac{3\pi}{L}(\pos[2]-\nicefrac{1}{2})}\\
    \pos[2] + \epsilon L
    \sin\lr{\nicefrac{4\pi}{L}(\pos[1]-\nicefrac{1}{2})}\cos\lr{\nicefrac{\pi}{L}(\pos[2]-\nicefrac{1}{2})}\\
  \end{matrix} \right],\\[2ex]
  \mathrm{3D}:& \ \pos = \left[ \begin{matrix}
    \pos[1] + \epsilon L
    \cos\lr{\nicefrac{\pi}{L}(\pos[1]-\nicefrac{1}{2})}\sin\lr{\nicefrac{4\pi}{L}(\pos[2]-\nicefrac{1}{2})}\cos\lr{\nicefrac{\pi}{L}(\pos[3]-\nicefrac{1}{2})}\\
    \pos[2] + \epsilon L
    \cos\lr{\nicefrac{3\pi}{L}(\pos[1]-\nicefrac{1}{2})}\cos\lr{\nicefrac{\pi}{L}(\pos[2]-\nicefrac{1}{2})}\cos\lr{\nicefrac{\pi}{L}(\pos[3]-\nicefrac{1}{2})}\\
    \pos[3] + \epsilon L \cos\lr{\nicefrac{\pi}{L}(\pos[1]-\nicefrac{1}{2})}\cos\lr{\nicefrac{2\pi}{L}(\pos[2]-\nicefrac{1}{2})}\cos\lr{\nicefrac{\pi}{L}(\pos[3]-\nicefrac{1}{2})}
  \end{matrix} \right]
\end{align*}
where $L=1$ denotes the length of the domain and $\epsilon$ chosen as $\epsilon=\nicefrac{1}{16}$, following~\cite{Montoya2024b}.
The curvilinear elements are approximated by a polynomial degree of $\ppngeo = 2$.
The chosen error norm is the discrete $L_2$-error of the density at $t=0.1$.
The results in~\cref{fig:validation:conv} highlight the spatial convergence properties of the entropy-stable DGSEM on different curvilinear elements.
Note that additionally uncurved pyramids are considered, since convergence cannot be guaranteed for curved pyramids.
This is, however, a common phenomenon in the literature~\cite{Chan2017}.
In order to avoid a stagnation in convergence on the hybrid grid, the node position of the pyramids are not perturbed.

\begin{figure}[htbp]
  \centering
  \includegraphics[width=\linewidth]{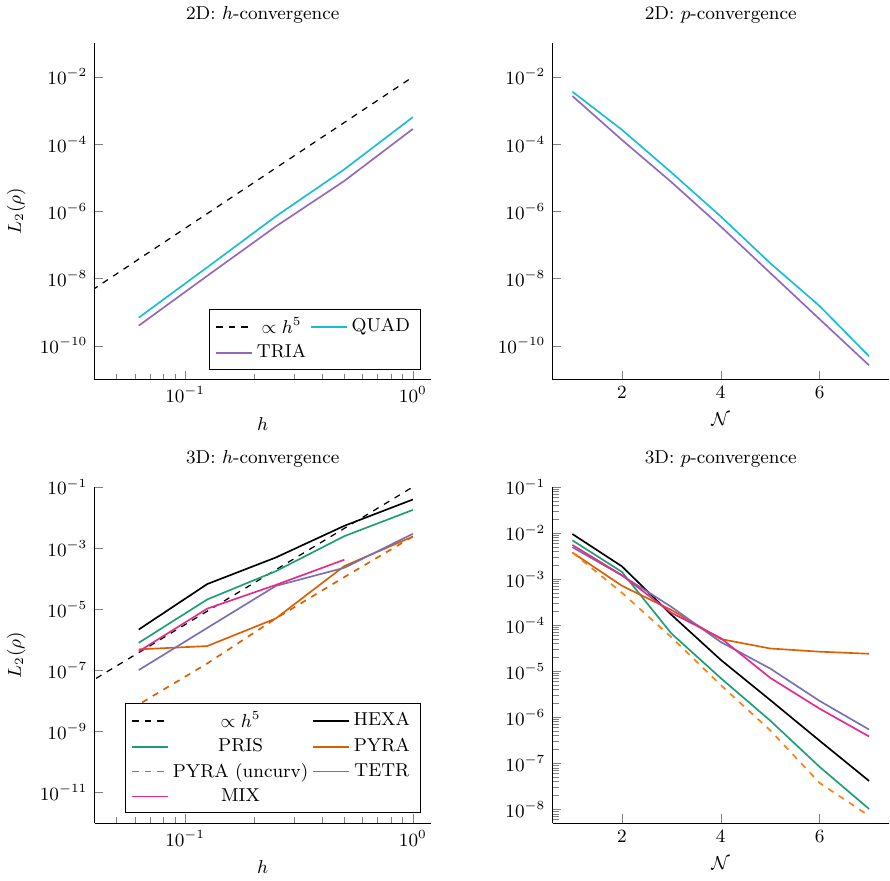}
  \caption{Validation of the spatial discretization of the entropy-stable DGSEM on curvilinear meshes for 2D (top) and 3D (bottom). Left: $p$-convergence on a $4^3$ grid with $\ppn \in [1,7]$. Right: $h$-convergence for $\ppn=4$. In 3D, each hexahedral element is either kept as one hexahedron (\texttt{HEXA}) or split into two prisms (\texttt{PRIS}), six pyramids (\texttt{PYRA}), or six tetrahedrons (\texttt{TETR}). In 2D, each quadrilateral element is either kept as one quadrilateral (\texttt{QUAD}) element or split into two triangular (\texttt{TRIA}) elements. In addition, a curvilinear hybrid mesh (MIX) composed of all four element types and an unperturbed pyramid mesh (\texttt{PYRA} (uncurv)) are considered.}
  \label{fig:validation:conv}
\end{figure}

\begin{figure}
  \includegraphics[width=\linewidth]{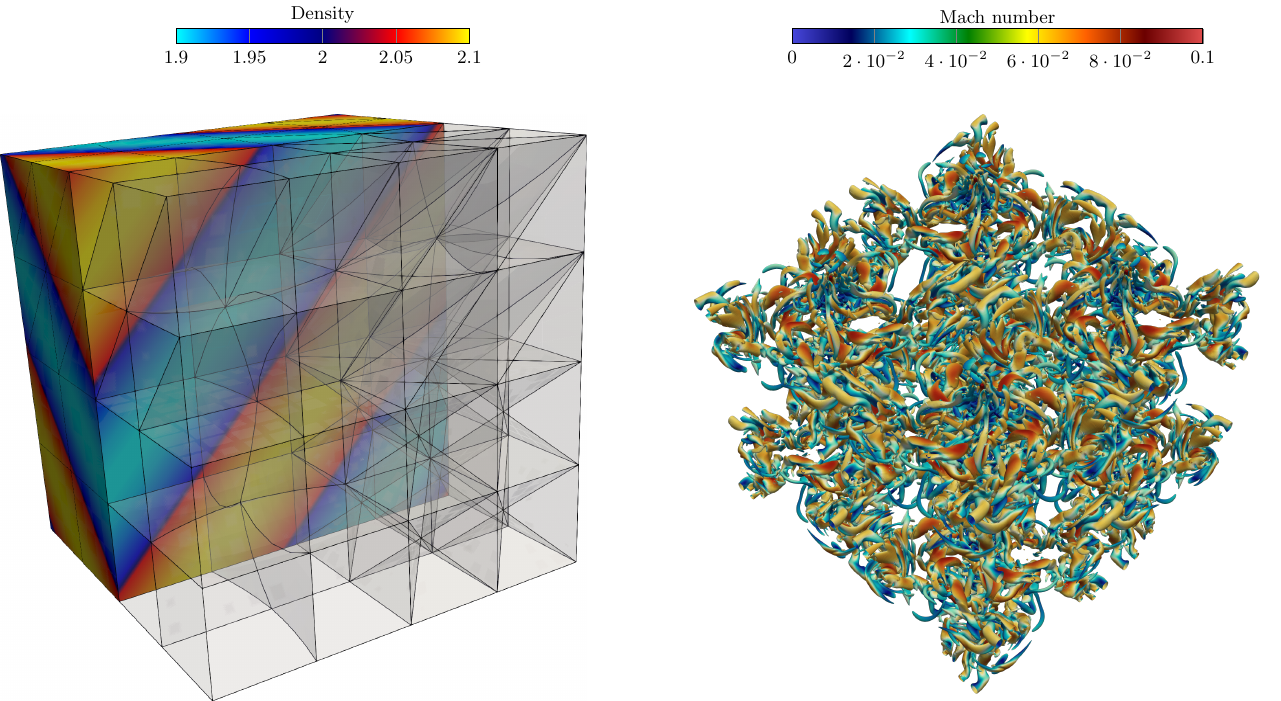}
  \caption{Left: Density of the convergence test with a visualization of the mixed, curvilinear mesh with \num{248} elements.
    Right: Instantaneous flow field of the weakly compressible TGV case with $M = 0.1$ and $Re = 1600$ at time $t = 10$ using
  around \num{1.34E8} DOF, visualized by iso-surfaces of the Q-criterion colored by Mach number, taken from~\cite{Kempf2024}.}
  \label{fig:convtest_tgv}
\end{figure}

\subsection{Weakly Compressible Taylor--Green Vortex}
The Taylor-Green vortex (TGV)~\cite{Taylor1937} is a popular test case for the validation of entropy conservation and
stability. The initial conditions are given as
\begin{align}
  \vel(\pos,t=0) &= \left[ \begin{matrix} \phantom{-} u_0\sin(\nicefrac{\pos[1]}{L})\cos(\nicefrac{\pos[2]}{L})\cos(\nicefrac{\pos[3]}{L}) \\
                                                     -u_0\cos(\nicefrac{\pos[1]}{L})\cos(\nicefrac{\pos[2]}{L})\cos(\nicefrac{\pos[3]}{L}) \\ 0
                                     \end{matrix} \right], \\
  p(\pos,t=0)    &= p_0 + \nicefrac{\rho_0
      u_0^2}{16}\lr{\cos(\nicefrac{2\pos[1]}{L})+\cos(\nicefrac{2\pos[2]}{L})}\lr{2+\cos(\nicefrac{2\pos[3]}{L})}, \nonumber
  \label{eq:TGV_init}
\end{align}
on a computational domain of size $\Omega = [0,2\pi L]^3$ with periodic boundary conditions.
Here, $L=1$ denotes the characteristic length, $u_0$ is the magnitude of the initial velocity fluctuations, and $\rho_0$ is the
reference density.
The initial pressure $p_0$ is chosen consistent to the prescribed Mach number $\mathrm{M}_0 = u_0 \sqrt{\nicefrac{\rho_0}{\gamma
p_0}}$. The desired Reynolds number, $\mathrm{Re} = \nicefrac{\rho_0 u_0 L}{\mu_0}$, is obtained by adjusting the dynamic viscosity $\mu_0$.
The density $\rho$ is obtained from the equation of state of a perfect gas $\rho = \nicefrac{p}{RT}$ with $T(\pos,t=0)=T_0$.

\subsubsection{Entropy Conservation and Stability}
First, the entropy conservation and stability properties are investigated using the inviscid TGV in the weakly compressible regime with $M = 0.1$.
A common metric to check entropy conservation is the change of discrete integral entropy over time, given as
\begin{align}
  \Delta_{\mathbb{S}} = \entropy(t)-\entropy(t=0), \quad \entropy(t) = \sum\limits_{\iel=1}^{\nel}
  \projE{\entropy^{(\iel)}(t),\J^{(\iel)}}.
\end{align}
As previously described, a hexahedral grid of size $4^3$ is utilized, where each hexahedral element is split into the corresponding element type.
In addition, a mixed elements grid is considered with polynomial degrees $\ppn = \lbrace 4,5,6,7,8 \rbrace$, resulting in $\num{16000}$ up to $\num{279936}$ DOFs.
The entropy conserving two-point flux of \citet{Chandrashekar2013} (CH) is used, together with a Lax--Friedrichs numerical flux function across adjacent elements.
The temporal evolution of the integral entropy error $\Delta_{\mathbb{S}}$, cf.~\cref{fig:tgv}, illustrates that the CH flux recovers the entropy conservation property for the case without additional surface dissipation, i.e., only the central part of the Riemann solver is employed. Otherwise, the scheme is entropy-stable.

\begin{figure}[htbp!]
  \centering
  \includegraphics[width=\linewidth]{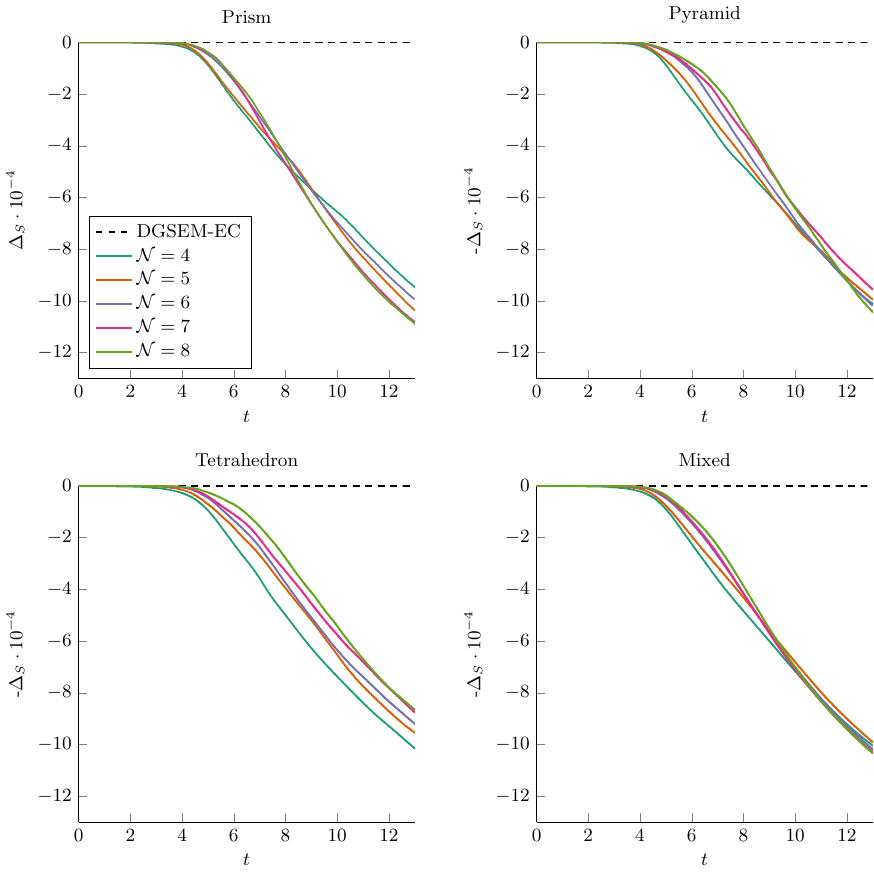}
  \caption{Temporal evolution of the integral entropy conservation errors for the Euler equations using the weakly compressible TGV at $M=0.1$. The grid is composed of prisms, pyramids, tetrahedrons, and a mixed grid (from left to right and top to bottom) with ($\ppn = \lbrace 4,5,6,7,8 \rbrace$) and without (DGSEM-EC) additional surface dissipation.}%
  \label{fig:tgv}
\end{figure}

\subsubsection{Kinetic Energy and Dissipation Rate}
To assess the accuracy of the numerical scheme, the (viscous) Taylor-Green vortex at $\mathrm{Re}=1600$ in the weakly compressible
regime with $\mathrm{M}=0.1$ is considered, cf.~\cref{fig:convtest_tgv} (right).
Two common metrics are the instantaneous kinetic energy $e_k$ and the solenoidal component $\epsilon_S$ of the viscous dissipation rate, defined as
\begin{align}
  e_k        &= \frac{1}{2 \rho_0 u_0^2 \abs{\Omega}} \int_{\Omega} \rho \vel \cdot \vel~d\Omega, \\
  \epsilon_S &= \frac{L^2}{\mathrm{Re} u_0^2 \abs{\Omega}} \int_{\Omega} \boldsymbol{\omega} \cdot \boldsymbol{\omega}~d \Omega,
\end{align}
respectively.
Here, $\boldsymbol{\omega}=\nabla \times \vel$ denotes the vorticity.
The solenoidal dissipation rate can be related to the amount of dissipated kinetic energy by small scale turbulence.
The results by~\citet{Debonis2013} serve as a reference.
The convergence of the instantaneous kinetic energy and the solenoidal dissipation rate during grid refinement to the reference is illustrated in~\cref{fig:tgv_incomp} for the mixed elements grid.
The grid sequence ranges from $4^3$ to $32^3$ (in terms of initial hexahedral elements), resulting in a total of $\num{1.31e5}$ up to $\num{66.56e6}$ DOF for $\ppn = 7$.
An illustration of the instantaneous flow field at $t=10$ is depicted in \cref{fig:convtest_tgv} (right).
\begin{figure}[htbp!]
  \centering
  \includegraphics[width=\linewidth]{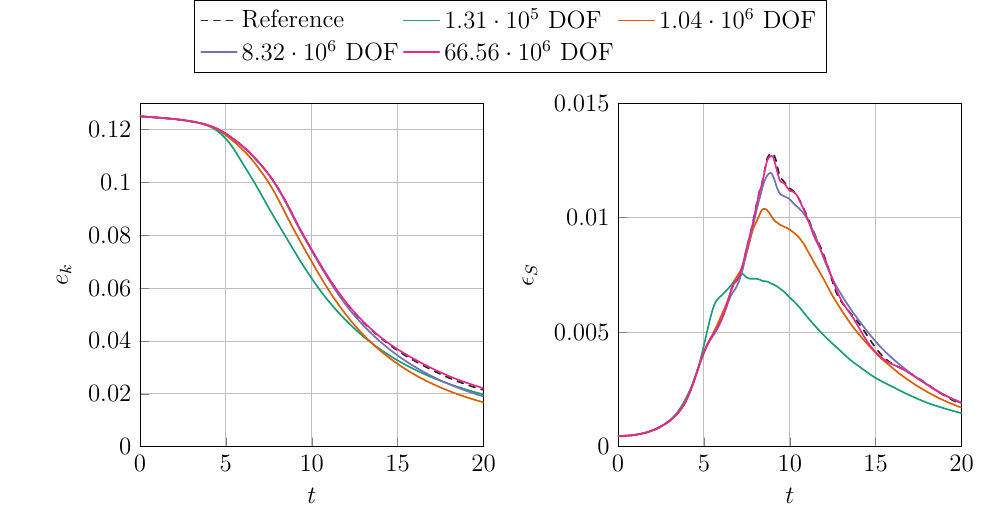}
  \caption{Weakly compressible TGV at $M=0.1$ for $\ppn = 7$ using a hybrid grid. Left: Temporal evolution of the instantaneous kinetic energy. Right: Temporal evolution of the solenoidal dissipation rate.
  The results of~\citet{Debonis2013} serve as a reference.}
  \label{fig:tgv_incomp}
\end{figure}
\subsection{Flow Around the Common Research Model}

The applicability of the proposed scheme to more complex, large-scale problems is demonstrated using the flow around the NASA Common Research Model (CRM), a well-established aerodynamic benchmark geometry developed for code validation and comparative studies in computational fluid dynamics (CFD) \cite{Vassberg2008}.
The CRM, representative of a modern transonic commercial transport aircraft, was designed to mitigate inconsistencies in drag prediction among CFD solvers.
The geometry used in this example is based on the standard CRM CAD model\footnote{\url{https://commonresearchmodel.larc.nasa.gov}},
and it is meshed using \num{147763} second-order curved hybrid elements.

It is imperative to acknowledge that this simulation solely serves as a proof-of-principle; thus, no detailed boundary-layer resolution is included, and only the aircraft's geometry is sufficiently captured. The chosen configuration mimics taxiing conditions on the runway at ambient atmospheric settings, assuming an inviscid flow regime with slip wall boundary conditions. This simplified setup results in an initial density of $\rho=\SI{1.225}{\kilogram\per\meter\cubed}$, a pressure of $p=\SI{101325}{\pascal}$, and a velocity vector of $\vel=\arr{\SI{15}{\meter\per\second},\SI{0}{\meter\per\second},\SI{0}{\meter\per\second}}$. Farfield boundary conditions are applied at the outer boundaries of the computational domain to simulate the undisturbed ambient environment during taxiing, ensuring well-posedness of the inviscid flow problem.

In~\cref{fig:crm}, the second-order unstructured surface grid together with the instantaneous distribution of the surface pressure at $t=\SI{1.2}{\second}$ using $\ppn=3$ are illustrated.
The flow acceleration following the stagnation point on the tip of the aircraft as well as on the wing surface are clearly captured and visible as regions of reduced static pressure.

\begin{figure}[htbp!]
  \centering
  \pgfkeys{/pgf/number format/.cd,1000 sep={\,}}
  \begin{tikzpicture}
  \scriptsize
  \begin{axis}[
    hide axis,
    width=\linewidth,
    scale only axis,
    enlargelimits=false,
    axis equal image,
    colormap name=viridis,
    colorbar horizontal,
    point meta min= 10200,
    point meta max=10350,
    colorbar style={
        width=0.25*\pgfkeysvalueof{/pgfplots/parent axis width},
        height=0.015*\pgfkeysvalueof{/pgfplots/parent axis width},
        color=black,
        scaled x ticks = false,
        at={(rel axis cs:0.7,0.13)},
        anchor=south west,
        title={Pressure in \si{\pascal}},
        title style={yshift=-7pt},
      }]
    \addplot graphics [xmin=0,xmax=4000,ymin=0,ymax=2000] {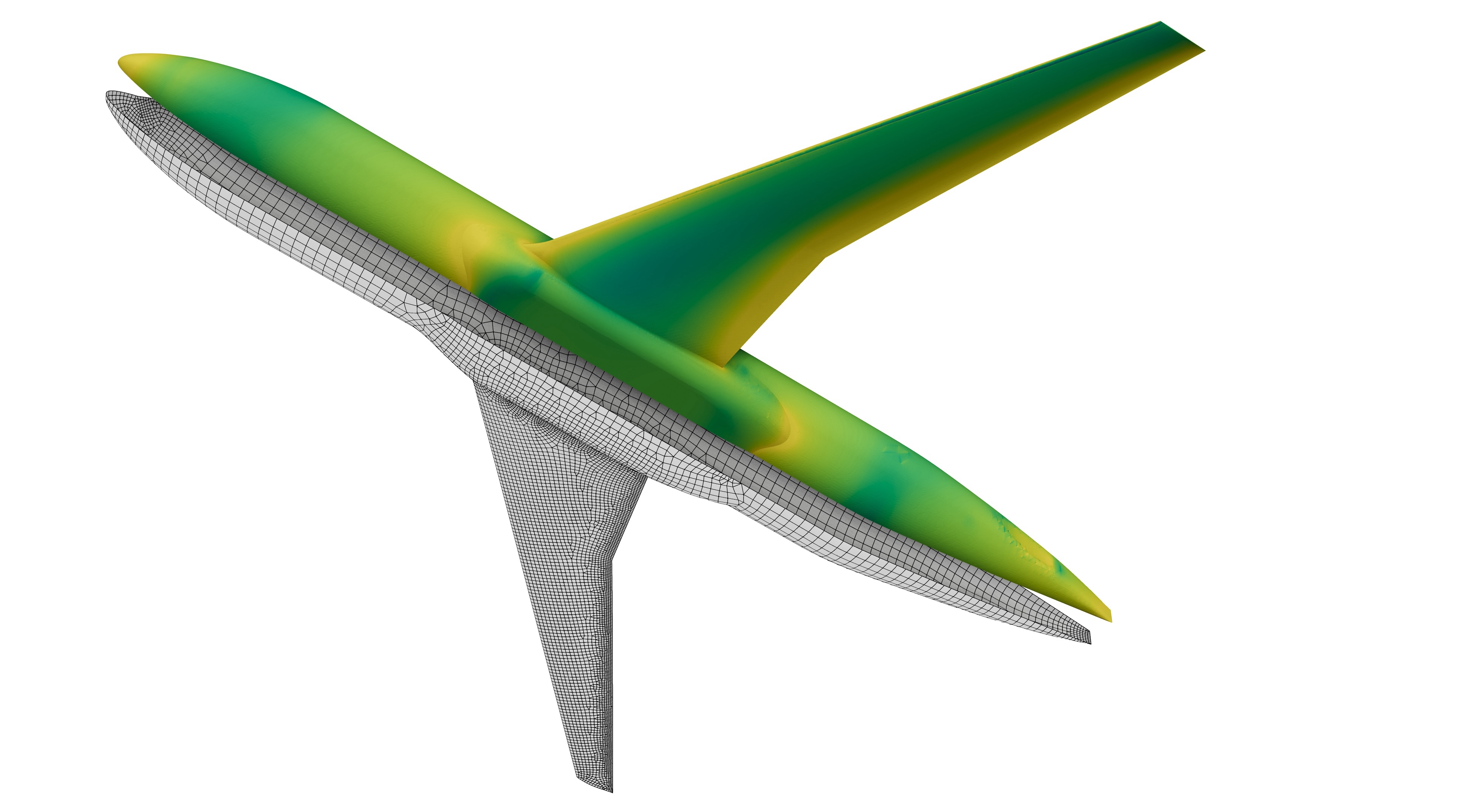};
  \end{axis}
  \end{tikzpicture}
  \caption{NASA common research model (CRM). Instantaneous distribution of surface pressure at $t=1.2$ with the second-order unstructured surface grid.}
  \label{fig:crm}
\end{figure}

\section{Conclusion}
\label{sec:conclusion}

High-order methods such as the \DG~scheme are promising candidates for the numerical simulation of complex, turbulent flows due to their
geometric flexibility as well as their favorable dissipation and dispersion properties compared to low-order schemes.
However, high-order methods are prone to instabilities in the presence of underresolved flow features such that adequate
stabilization techniques are mandatory to provide a robust numerical scheme.
Here, the entropy-stable \DG~formulation based on diagonal-norm SBP operators serves as a notable example.
The efficient construction of multidimensional SBP operators of arbitrary order is challenging.
One proposed solution to this issue is to use of a tensor-product-based approach such as the \DGSEM, which in its original form, is
however limited to hexahedral elements.
\newline
The key contribution of this work is the formulation of an entropy-stable \DGSEM~of arbitrary order on curvilinear hybrid meshes, including hexahedrons, prisms, pyramids, and tetrahedrons.
The extension of the \DGSEM~to more complex element shapes is achieved by means of a collapsed coordinate transformation.
Legendre--Gauss quadrature nodes are utilized as collocation points in conjunction with a generalized SBP operator and
entropy-projected variables.
To circumvent the penalizing time step restriction imposed by the collapsing, modal rather than nodal degrees of freedom are evolved in time, thereby relying on a memory-efficient weight-adjusted approximation to the inverse of the mass matrix.
An extension of the purely hyperbolic operator to hyperbolic-parabolic problems is accomplished by means of a lifting procedure.
The proposed scheme was thoroughly validated, starting with a series of numerical experiments and progressing to a more
advanced large-scale application: the viscous flow around the common research model.
These studies demonstrated the free-stream preservation, convergence, and entropy conservation / entropy stability properties, as well as
the applicability to real-world problems.
In the future, the numerical scheme will be applied to more complex, turbulent flows, and its potential to enhance efficiency will be explored further by porting the scheme to GPU-accelerated systems.

\section*{Acknowledgements}
This work was funded by the European Union and has received funding from the European High Performance Computing Joint Undertaking (JU) and Sweden, Germany, Spain, Greece, and Denmark under grant agreement No 101093393.
The research presented in this paper was funded in parts by Deutsche Forschungsgemeinschaft (DFG, German Research
Foundation) under Germany's Excellence Strategy - EXC 2075 - 390740016 and by the state of Baden-Württemberg under the project Aerospace 2050 MWK32-7531-49/13/7 "FLUTTER"/"QUASAR".
We acknowledge the support by the Stuttgart Center for Simulation Science (SimTech).
Further, we want to gratefully acknowledge funding by the DFG through SPP 2410 Hyperbolic Balance Laws in Fluid Mechanics: Complexity, Scales, Randomness (CoScaRa).
The authors gratefully acknowledge the support and the computing time on ``Hawk'' provided by the HLRS through the project ``hpcdg''.

\bibliographystyle{elsarticle-num-names}
\bibliography{references}

\begin{thebibliography}{79}
\expandafter\ifx\csname natexlab\endcsname\relax\def\natexlab#1{#1}\fi
\providecommand{\url}[1]{\texttt{#1}}
\providecommand{\href}[2]{#2}
\providecommand{\path}[1]{#1}
\providecommand{\DOIprefix}{doi:}
\providecommand{\ArXivprefix}{arXiv:}
\providecommand{\URLprefix}{URL: }
\providecommand{\Pubmedprefix}{pmid:}
\providecommand{\doi}[1]{\href{http://dx.doi.org/#1}{\path{#1}}}
\providecommand{\Pubmed}[1]{\href{pmid:#1}{\path{#1}}}
\providecommand{\bibinfo}[2]{#2}
\ifx\xfnm\relax \def\xfnm[#1]{\unskip,\space#1}\fi
\bibitem[{Gassner and Kopriva(2011)}]{Gassner2011}
\bibinfo{author}{G.~Gassner}, \bibinfo{author}{D.~A. Kopriva},
\newblock \bibinfo{title}{{A Comparison of the Dispersion and Dissipation
  Errors of Gauss and Gauss–Lobatto Discontinuous Galerkin Spectral Element
  Methods}},
\newblock \bibinfo{journal}{SIAM Journal on Scientific Computing}
  \bibinfo{volume}{33} (\bibinfo{year}{2011}) \bibinfo{pages}{2560--2579}.
  \DOIprefix\doi{10.1137/100807211}.
\bibitem[{Blind et~al.(2024)Blind, Kopper, Kempf, Kurz, Schwarz, Munz, and
  Beck}]{Blind2024a}
\bibinfo{author}{M.~Blind}, \bibinfo{author}{P.~Kopper},
  \bibinfo{author}{D.~Kempf}, \bibinfo{author}{M.~Kurz},
  \bibinfo{author}{A.~Schwarz}, \bibinfo{author}{C.-D. Munz},
  \bibinfo{author}{A.~Beck},
\newblock \bibinfo{title}{{Performance Improvements for Large-Scale Simulations
  using the Discontinuous Galerkin Framework FLEXI}},
\newblock in: \bibinfo{booktitle}{High Performance Computing in Science and
  Engineering '22}, \bibinfo{publisher}{Springer Nature Switzerland},
  \bibinfo{address}{Cham}, \bibinfo{year}{2024}, pp. \bibinfo{pages}{249--264}.
  \DOIprefix\doi{10.1007/978-3-031-46870-4_17}.
\bibitem[{Kurz et~al.(2025)Kurz, Kempf, Blind, Kopper, Offenh{\"{a}}user,
  Schwarz, Starr, Keim, and Beck}]{Kempf2024}
\bibinfo{author}{M.~Kurz}, \bibinfo{author}{D.~Kempf}, \bibinfo{author}{M.~P.
  Blind}, \bibinfo{author}{P.~Kopper}, \bibinfo{author}{P.~Offenh{\"{a}}user},
  \bibinfo{author}{A.~Schwarz}, \bibinfo{author}{S.~Starr},
  \bibinfo{author}{J.~Keim}, \bibinfo{author}{A.~Beck},
\newblock \bibinfo{title}{{GAL{\AE}XI: Solving complex compressible flows with
  high-order discontinuous Galerkin methods on accelerator-based systems}},
\newblock \bibinfo{journal}{Computer Physics Communications}
  \bibinfo{volume}{306} (\bibinfo{year}{2025}) \bibinfo{pages}{109388}.
  \DOIprefix\doi{10.1016/j.cpc.2024.109388}.
\bibitem[{Mavriplis(1994)}]{Mavriplis1994}
\bibinfo{author}{C.~Mavriplis},
\newblock \bibinfo{title}{{Adaptive mesh strategies for the spectral element
  method}},
\newblock \bibinfo{journal}{Computer Methods in Applied Mechanics and
  Engineering} \bibinfo{volume}{116} (\bibinfo{year}{1994})
  \bibinfo{pages}{77--86}. \DOIprefix\doi{10.1016/S0045-7825(94)80010-3}.
\bibitem[{Mossier et~al.(2022)Mossier, Beck, and Munz}]{Mossier2022}
\bibinfo{author}{P.~Mossier}, \bibinfo{author}{A.~Beck}, \bibinfo{author}{C.~D.
  Munz},
\newblock \bibinfo{title}{{A p-Adaptive Discontinuous Galerkin Method with
  hp-Shock Capturing}},
\newblock \bibinfo{journal}{Journal of Scientific Computing}
  \bibinfo{volume}{91} (\bibinfo{year}{2022}) \bibinfo{pages}{1--36}.
  \DOIprefix\doi{10.1007/s10915-022-01770-6}.
\bibitem[{Cantwell et~al.(2015)Cantwell, Moxey, Comerford, Bolis, Rocco,
  Mengaldo, {De Grazia}, Yakovlev, Lombard, Ekelschot, Jordi, Xu, Mohamied,
  Eskilsson, Nelson, Vos, Biotto, Kirby, and Sherwin}]{Cantwell2015}
\bibinfo{author}{C.~D. Cantwell}, \bibinfo{author}{D.~Moxey},
  \bibinfo{author}{A.~Comerford}, \bibinfo{author}{A.~Bolis},
  \bibinfo{author}{G.~Rocco}, \bibinfo{author}{G.~Mengaldo},
  \bibinfo{author}{D.~{De Grazia}}, \bibinfo{author}{S.~Yakovlev},
  \bibinfo{author}{J.~E. Lombard}, \bibinfo{author}{D.~Ekelschot},
  \bibinfo{author}{B.~Jordi}, \bibinfo{author}{H.~Xu},
  \bibinfo{author}{Y.~Mohamied}, \bibinfo{author}{C.~Eskilsson},
  \bibinfo{author}{B.~Nelson}, \bibinfo{author}{P.~Vos},
  \bibinfo{author}{C.~Biotto}, \bibinfo{author}{R.~M. Kirby},
  \bibinfo{author}{S.~J. Sherwin},
\newblock \bibinfo{title}{{Nektar++: An open-source spectral/hp element
  framework}},
\newblock \bibinfo{journal}{Computer Physics Communications}
  \bibinfo{volume}{192} (\bibinfo{year}{2015}) \bibinfo{pages}{205--219}.
  \DOIprefix\doi{10.1016/j.cpc.2015.02.008}.
\bibitem[{Witherden et~al.(2015)Witherden, Vermeire, and
  Vincent}]{Witherden2015}
\bibinfo{author}{F.~D. Witherden}, \bibinfo{author}{B.~C. Vermeire},
  \bibinfo{author}{P.~E. Vincent},
\newblock \bibinfo{title}{{Heterogeneous computing on mixed unstructured grids
  with PyFR}},
\newblock \bibinfo{journal}{Computers and Fluids} \bibinfo{volume}{120}
  (\bibinfo{year}{2015}) \bibinfo{pages}{173--186}.
  \DOIprefix\doi{10.1016/j.compfluid.2015.07.016}.
\bibitem[{Krais et~al.(2021)Krais, Beck, Bolemann, Frank, Flad, Gassner,
  Hindenlang, Hoffmann, Kuhn, Sonntag, and Munz}]{Krais2019}
\bibinfo{author}{N.~Krais}, \bibinfo{author}{A.~Beck},
  \bibinfo{author}{T.~Bolemann}, \bibinfo{author}{H.~Frank},
  \bibinfo{author}{D.~Flad}, \bibinfo{author}{G.~Gassner},
  \bibinfo{author}{F.~Hindenlang}, \bibinfo{author}{M.~Hoffmann},
  \bibinfo{author}{T.~Kuhn}, \bibinfo{author}{M.~Sonntag},
  \bibinfo{author}{C.-D. Munz},
\newblock \bibinfo{title}{{FLEXI: A high order discontinuous Galerkin framework
  for hyperbolic-parabolic conservation laws}},
\newblock \bibinfo{journal}{Computers {\&} Mathematics with Applications}
  \bibinfo{volume}{81} (\bibinfo{year}{2021}) \bibinfo{pages}{186--219}.
  \DOIprefix\doi{10.1016/j.camwa.2020.05.004}.
\bibitem[{Schlottke-Lakemper et~al.(2021)Schlottke-Lakemper, Gassner, Ranocha,
  Winters, and Chan}]{Schlottkelakemper2020}
\bibinfo{author}{M.~Schlottke-Lakemper}, \bibinfo{author}{G.~J. Gassner},
  \bibinfo{author}{H.~Ranocha}, \bibinfo{author}{A.~R. Winters},
  \bibinfo{author}{J.~Chan}, \bibinfo{title}{{T}rixi.jl: {A}daptive high-order
  numerical simulations of hyperbolic {PDE}s in {J}ulia},
  \bibinfo{howpublished}{\url{https://github.com/trixi-framework/Trixi.jl}},
  \bibinfo{year}{2021}. \DOIprefix\doi{10.5281/zenodo.3996439}.
\bibitem[{Parsani et~al.(2021)Parsani, Boukharfane, Nolasco, {Del Rey
  Fern{\'{a}}ndez}, Zampini, Hadri, and Dalcin}]{Parsani2021}
\bibinfo{author}{M.~Parsani}, \bibinfo{author}{R.~Boukharfane},
  \bibinfo{author}{I.~R. Nolasco}, \bibinfo{author}{D.~C. {Del Rey
  Fern{\'{a}}ndez}}, \bibinfo{author}{S.~Zampini}, \bibinfo{author}{B.~Hadri},
  \bibinfo{author}{L.~Dalcin},
\newblock \bibinfo{title}{{High-order accurate entropy-stable discontinuous
  collocated Galerkin methods with the summation-by-parts property for
  compressible CFD frameworks: Scalable SSDC algorithms and flow solver}},
\newblock \bibinfo{journal}{Journal of Computational Physics}
  \bibinfo{volume}{424} (\bibinfo{year}{2021}) \bibinfo{pages}{109844}.
  \DOIprefix\doi{10.1016/j.jcp.2020.109844}.
\bibitem[{Ferrer et~al.(2023)Ferrer, Rubio, Ntoukas, Laskowski, Mari{\~{n}}o,
  Colombo, Mateo-Gab{\'{i}}n, Marbona, {Manrique de Lara}, Huergo, Manzanero,
  Rueda-Ram{\'{i}}rez, Kopriva, and Valero}]{Ferrer2023}
\bibinfo{author}{E.~Ferrer}, \bibinfo{author}{G.~Rubio},
  \bibinfo{author}{G.~Ntoukas}, \bibinfo{author}{W.~Laskowski},
  \bibinfo{author}{O.~Mari{\~{n}}o}, \bibinfo{author}{S.~Colombo},
  \bibinfo{author}{A.~Mateo-Gab{\'{i}}n}, \bibinfo{author}{H.~Marbona},
  \bibinfo{author}{F.~{Manrique de Lara}}, \bibinfo{author}{D.~Huergo},
  \bibinfo{author}{J.~Manzanero}, \bibinfo{author}{A.~Rueda-Ram{\'{i}}rez},
  \bibinfo{author}{D.~Kopriva}, \bibinfo{author}{E.~Valero},
\newblock \bibinfo{title}{{HORSES3D: A high-order discontinuous Galerkin solver
  for flow simulations and multi-physics applications}},
\newblock \bibinfo{journal}{Computer Physics Communications}
  \bibinfo{volume}{287} (\bibinfo{year}{2023}) \bibinfo{pages}{108700}.
  \DOIprefix\doi{10.1016/j.cpc.2023.108700}.
\bibitem[{Persson and Peraire(2012)}]{Persson2012}
\bibinfo{author}{P.-O. Persson}, \bibinfo{author}{J.~Peraire},
  \bibinfo{title}{{Sub-Cell Shock Capturing for Discontinuous Galerkin
  Methods}}, \bibinfo{year}{2012}. \DOIprefix\doi{10.2514/6.2006-112}.
\bibitem[{Dzanic and Witherden(2022)}]{Dzanic2022}
\bibinfo{author}{T.~Dzanic}, \bibinfo{author}{F.~Witherden},
\newblock \bibinfo{title}{{Positivity-preserving entropy-based adaptive
  filtering for discontinuous spectral element methods}},
\newblock \bibinfo{journal}{Journal of Computational Physics}
  \bibinfo{volume}{468} (\bibinfo{year}{2022}) \bibinfo{pages}{111501}.
  \DOIprefix\doi{10.1016/j.jcp.2022.111501}.
\bibitem[{Kuzmin(2021)}]{Kuzmin2021}
\bibinfo{author}{D.~Kuzmin},
\newblock \bibinfo{title}{{A new perspective on flux and slope limiting in
  discontinuous Galerkin methods for hyperbolic conservation laws}},
\newblock \bibinfo{journal}{Computer Methods in Applied Mechanics and
  Engineering} \bibinfo{volume}{373} (\bibinfo{year}{2021})
  \bibinfo{pages}{113569}. \DOIprefix\doi{10.1016/j.cma.2020.113569}.
\bibitem[{Gassner and Beck(2012)}]{Gassner2012}
\bibinfo{author}{G.~J. Gassner}, \bibinfo{author}{A.~D. Beck},
\newblock \bibinfo{title}{On the accuracy of high-order discretizations for
  underresolved turbulence simulations},
\newblock \bibinfo{journal}{Theoretical and Computational Fluid Dynamics}
  \bibinfo{volume}{27} (\bibinfo{year}{2012}) \bibinfo{pages}{221--237}.
  \DOIprefix\doi{10.1007/s00162-011-0253-7}.
\bibitem[{Sonntag and Munz(2014)}]{Sonntag14}
\bibinfo{author}{M.~Sonntag}, \bibinfo{author}{C.-D. Munz},
\newblock \bibinfo{title}{{Shock Capturing for Discontinuous Galerkin Methods
  using Finite Volume Subcells}},
\newblock in: \bibinfo{booktitle}{Finite Volumes for Complex Applications
  VII-Elliptic, Parabolic and Hyperbolic Problems},
  \bibinfo{publisher}{Springer: Berlin}, \bibinfo{year}{2014}, pp.
  \bibinfo{pages}{945--953}. \DOIprefix\doi{10.1007/978-3-319-05591-6_96}.
\bibitem[{Dumbser and Loub{\`{e}}re(2016)}]{Dumbser2016}
\bibinfo{author}{M.~Dumbser}, \bibinfo{author}{R.~Loub{\`{e}}re},
\newblock \bibinfo{title}{{A simple robust and accurate a posteriori sub-cell
  finite volume limiter for the discontinuous Galerkin method on unstructured
  meshes}},
\newblock \bibinfo{journal}{Journal of Computational Physics}
  \bibinfo{volume}{319} (\bibinfo{year}{2016}) \bibinfo{pages}{163--199}.
  \DOIprefix\doi{10.1016/j.jcp.2016.05.002}.
\bibitem[{Huerta et~al.(2012)Huerta, Casoni, and Peraire}]{Huerta2012}
\bibinfo{author}{A.~Huerta}, \bibinfo{author}{E.~Casoni},
  \bibinfo{author}{J.~Peraire},
\newblock \bibinfo{title}{{A simple shock‐capturing technique for
  high‐order discontinuous Galerkin methods}},
\newblock \bibinfo{journal}{International Journal for Numerical Methods in
  Fluids} \bibinfo{volume}{69} (\bibinfo{year}{2012})
  \bibinfo{pages}{1614--1632}. \DOIprefix\doi{10.1002/fld.2654}.
\bibitem[{Hennemann et~al.(2021)Hennemann, Rueda-Ram{\'{i}}rez, Hindenlang, and
  Gassner}]{Hennemann2021}
\bibinfo{author}{S.~Hennemann}, \bibinfo{author}{A.~M. Rueda-Ram{\'{i}}rez},
  \bibinfo{author}{F.~J. Hindenlang}, \bibinfo{author}{G.~J. Gassner},
\newblock \bibinfo{title}{{A provably entropy stable subcell shock capturing
  approach for high order split form DG for the compressible Euler equations}},
\newblock \bibinfo{journal}{Journal of Computational Physics}
  \bibinfo{volume}{426} (\bibinfo{year}{2021}) \bibinfo{pages}{109935}.
  \DOIprefix\doi{10.1016/j.jcp.2020.109935}.
\bibitem[{Chan(2018)}]{Chan2018}
\bibinfo{author}{J.~Chan},
\newblock \bibinfo{title}{{On discretely entropy conservative and entropy
  stable discontinuous Galerkin methods}},
\newblock \bibinfo{journal}{Journal of Computational Physics}
  \bibinfo{volume}{362} (\bibinfo{year}{2018}) \bibinfo{pages}{346--374}.
  \DOIprefix\doi{10.1016/j.jcp.2018.02.033}.
\bibitem[{Tadmor(1987)}]{Tadmor1987}
\bibinfo{author}{E.~Tadmor},
\newblock \bibinfo{title}{{The numerical viscosity of entropy stable schemes
  for systems of conservation laws. I}},
\newblock \bibinfo{journal}{Mathematics of Computation} \bibinfo{volume}{49}
  (\bibinfo{year}{1987}) \bibinfo{pages}{91--103}.
  \DOIprefix\doi{10.1090/S0025-5718-1987-0890255-3}.
\bibitem[{LeFloch et~al.(2002)LeFloch, Mercier, and Rohde}]{LeFloch2002}
\bibinfo{author}{P.~G. LeFloch}, \bibinfo{author}{J.~M. Mercier},
  \bibinfo{author}{C.~Rohde},
\newblock \bibinfo{title}{{Fully Discrete, Entropy Conservative Schemes of
  Arbitrary Order}},
\newblock \bibinfo{journal}{SIAM Journal on Numerical Analysis}
  \bibinfo{volume}{40} (\bibinfo{year}{2002}) \bibinfo{pages}{1968--1992}.
  \DOIprefix\doi{10.1137/S003614290240069X}.
\bibitem[{Fisher and Carpenter(2013)}]{Fisher2013}
\bibinfo{author}{T.~C. Fisher}, \bibinfo{author}{M.~H. Carpenter},
\newblock \bibinfo{title}{High-order entropy stable finite difference schemes
  for nonlinear conservation laws: Finite domains},
\newblock \bibinfo{journal}{J. Comput. Phys.} \bibinfo{volume}{252}
  (\bibinfo{year}{2013}) \bibinfo{pages}{518--557}.
\bibitem[{Carpenter et~al.(2014)Carpenter, Fisher, Nielsen, and
  Frankel}]{Carpenter2014}
\bibinfo{author}{M.~H. Carpenter}, \bibinfo{author}{T.~C. Fisher},
  \bibinfo{author}{E.~J. Nielsen}, \bibinfo{author}{S.~H. Frankel},
\newblock \bibinfo{title}{{Entropy Stable Spectral Collocation Schemes for the
  Navier--Stokes Equations: Discontinuous Interfaces}},
\newblock \bibinfo{journal}{SIAM J. Sci. Comput.} \bibinfo{volume}{36}
  (\bibinfo{year}{2014}) \bibinfo{pages}{B835--B867}.
  \DOIprefix\doi{10.1137/130932193}.
\bibitem[{Gassner(2013)}]{Gassner2013}
\bibinfo{author}{G.~J. Gassner},
\newblock \bibinfo{title}{{A Skew-Symmetric Discontinuous Galerkin Spectral
  Element Discretization and Its Relation to SBP-SAT Finite Difference
  Methods}},
\newblock \bibinfo{journal}{SIAM Journal on Scientific Computing}
  \bibinfo{volume}{35} (\bibinfo{year}{2013}) \bibinfo{pages}{A1233--A1253}.
  \DOIprefix\doi{10.1137/120890144}.
\bibitem[{Gassner et~al.(2016)Gassner, Winters, and Kopriva}]{Gassner2016}
\bibinfo{author}{G.~J. Gassner}, \bibinfo{author}{A.~R. Winters},
  \bibinfo{author}{D.~A. Kopriva},
\newblock \bibinfo{title}{{Split form nodal discontinuous Galerkin schemes with
  summation-by-parts property for the compressible Euler equations}},
\newblock \bibinfo{journal}{Journal of Computational Physics}
  \bibinfo{volume}{327} (\bibinfo{year}{2016}) \bibinfo{pages}{39--66}.
  \DOIprefix\doi{10.1016/j.jcp.2016.09.013}.
\bibitem[{{Del Rey Fern{\'{a}}ndez} et~al.(2014){Del Rey Fern{\'{a}}ndez},
  Boom, and Zingg}]{DelReyFernandez2014}
\bibinfo{author}{D.~C. {Del Rey Fern{\'{a}}ndez}}, \bibinfo{author}{P.~D.
  Boom}, \bibinfo{author}{D.~W. Zingg},
\newblock \bibinfo{title}{{A generalized framework for nodal first derivative
  summation-by-parts operators}},
\newblock \bibinfo{journal}{Journal of Computational Physics}
  \bibinfo{volume}{266} (\bibinfo{year}{2014}) \bibinfo{pages}{214--239}.
  \DOIprefix\doi{10.1016/j.jcp.2014.01.038}.
\bibitem[{Hicken et~al.(2016)Hicken, {Del Rey Fern{\'{a}}ndez}, and
  Zingg}]{Hicken2016}
\bibinfo{author}{J.~E. Hicken}, \bibinfo{author}{D.~C. {Del Rey
  Fern{\'{a}}ndez}}, \bibinfo{author}{D.~W. Zingg},
\newblock \bibinfo{title}{{Multidimensional Summation-by-Parts Operators:
  General Theory and Application to Simplex Elements}},
\newblock \bibinfo{journal}{SIAM Journal on Scientific Computing}
  \bibinfo{volume}{38} (\bibinfo{year}{2016}) \bibinfo{pages}{A1935--A1958}.
  \DOIprefix\doi{10.1137/15M1038360}.
\bibitem[{Chen and Shu(2017)}]{Chen2017}
\bibinfo{author}{T.~Chen}, \bibinfo{author}{C.~W. Shu},
\newblock \bibinfo{title}{{Entropy stable high order discontinuous Galerkin
  methods with suitable quadrature rules for hyperbolic conservation laws}},
\newblock \bibinfo{journal}{Journal of Computational Physics}
  \bibinfo{volume}{345} (\bibinfo{year}{2017}) \bibinfo{pages}{427--461}.
  \DOIprefix\doi{10.1016/j.jcp.2017.05.025}.
\bibitem[{Chan et~al.(2017)Chan, Hewett, and Warburton}]{Chan2017}
\bibinfo{author}{J.~Chan}, \bibinfo{author}{R.~J. Hewett},
  \bibinfo{author}{T.~Warburton},
\newblock \bibinfo{title}{{Weight-Adjusted Discontinuous Galerkin Methods:
  Curvilinear Meshes}},
\newblock \bibinfo{journal}{SIAM Journal on Scientific Computing}
  \bibinfo{volume}{39} (\bibinfo{year}{2017}) \bibinfo{pages}{A2395--A2421}.
  \DOIprefix\doi{10.1137/16M1089198}.
\bibitem[{Crean et~al.(2018)Crean, Hicken, {Del Rey Fern{\'{a}}ndez}, Zingg,
  and Carpenter}]{Crean2018}
\bibinfo{author}{J.~Crean}, \bibinfo{author}{J.~E. Hicken},
  \bibinfo{author}{D.~C. {Del Rey Fern{\'{a}}ndez}}, \bibinfo{author}{D.~W.
  Zingg}, \bibinfo{author}{M.~H. Carpenter},
\newblock \bibinfo{title}{{Entropy-stable summation-by-parts discretization of
  the Euler equations on general curved elements}},
\newblock \bibinfo{journal}{Journal of Computational Physics}
  \bibinfo{volume}{356} (\bibinfo{year}{2018}) \bibinfo{pages}{410--438}.
  \DOIprefix\doi{10.1016/j.jcp.2017.12.015}.
\bibitem[{Chan et~al.(2019)Chan, {Del Rey Fern{\'{a}}ndez}, and
  Carpenter}]{Chan2019}
\bibinfo{author}{J.~Chan}, \bibinfo{author}{D.~C. {Del Rey Fern{\'{a}}ndez}},
  \bibinfo{author}{M.~H. Carpenter},
\newblock \bibinfo{title}{{Efficient Entropy Stable Gauss Collocation
  Methods}},
\newblock \bibinfo{journal}{SIAM Journal on Scientific Computing}
  \bibinfo{volume}{41} (\bibinfo{year}{2019}) \bibinfo{pages}{A2938--A2966}.
  \DOIprefix\doi{10.1137/18M1209234}.
\bibitem[{{Tianheng Chen} and {Chi-Wang Shu}(2020)}]{TianhengChen2020}
\bibinfo{author}{T.~C. {Tianheng Chen}}, \bibinfo{author}{C.-W.~S. {Chi-Wang
  Shu}},
\newblock \bibinfo{title}{{Review of Entropy Stable Discontinuous Galerkin
  Methods for Systems of Conservation Laws on Unstructured Simplex Meshes}},
\newblock \bibinfo{journal}{CSIAM Transactions on Applied Mathematics}
  \bibinfo{volume}{1} (\bibinfo{year}{2020}) \bibinfo{pages}{1--52}.
  \DOIprefix\doi{10.4208/csiam-am.2020-0003}.
\bibitem[{Orszag(1980)}]{Orszag1980}
\bibinfo{author}{S.~A. Orszag},
\newblock \bibinfo{title}{{Spectral methods for problems in complex
  geometries}},
\newblock \bibinfo{journal}{Journal of Computational Physics}
  \bibinfo{volume}{37} (\bibinfo{year}{1980}) \bibinfo{pages}{70--92}.
  \DOIprefix\doi{10.1016/0021-9991(80)90005-4}.
\bibitem[{Worku et~al.(2024)Worku, Hicken, and Zingg}]{Worku2024}
\bibinfo{author}{Z.~A. Worku}, \bibinfo{author}{J.~E. Hicken},
  \bibinfo{author}{D.~W. Zingg},
\newblock \bibinfo{title}{{Quadrature Rules on Triangles and Tetrahedra for
  Multidimensional Summation-By-Parts Operators}},
\newblock \bibinfo{journal}{Journal of Scientific Computing}
  \bibinfo{volume}{101} (\bibinfo{year}{2024}) \bibinfo{pages}{24}.
  \DOIprefix\doi{10.1007/s10915-024-02656-5}.
\bibitem[{Montoya and Zingg(2024)}]{Montoya2024b}
\bibinfo{author}{T.~Montoya}, \bibinfo{author}{D.~W. Zingg},
\newblock \bibinfo{title}{{Efficient Tensor-Product Spectral-Element Operators
  with the Summation-by-Parts Property on Curved Triangles and Tetrahedra}},
\newblock \bibinfo{journal}{SIAM Journal on Scientific Computing}
  \bibinfo{volume}{46} (\bibinfo{year}{2024}) \bibinfo{pages}{A2270--A2297}.
  \DOIprefix\doi{10.1137/23M1573963}.
\bibitem[{D{\"{u}}rrw{\"{a}}chter et~al.(2021)D{\"{u}}rrw{\"{a}}chter, Kurz,
  Kopper, Kempf, Munz, and Beck}]{Durrwachter2021}
\bibinfo{author}{J.~D{\"{u}}rrw{\"{a}}chter}, \bibinfo{author}{M.~Kurz},
  \bibinfo{author}{P.~Kopper}, \bibinfo{author}{D.~Kempf},
  \bibinfo{author}{C.-D. Munz}, \bibinfo{author}{A.~Beck},
\newblock \bibinfo{title}{{An efficient sliding mesh interface method for
  high-order discontinuous Galerkin schemes}},
\newblock \bibinfo{journal}{Computers \& Fluids} \bibinfo{volume}{217}
  (\bibinfo{year}{2021}) \bibinfo{pages}{104825}.
  \DOIprefix\doi{10.1016/j.compfluid.2020.104825}.
\bibitem[{Blind et~al.(2023)Blind, Gao, Kempf, Kopper, Kurz, Schwarz, and
  Beck}]{Blind2023}
\bibinfo{author}{M.~Blind}, \bibinfo{author}{M.~Gao},
  \bibinfo{author}{D.~Kempf}, \bibinfo{author}{P.~Kopper},
  \bibinfo{author}{M.~Kurz}, \bibinfo{author}{A.~Schwarz},
  \bibinfo{author}{A.~Beck},
\newblock \bibinfo{title}{{Towards Exascale CFD Simulations Using the
  Discontinuous Galerkin Solver FLEXI}}  (\bibinfo{year}{2023}). \URLprefix
  \url{http://arxiv.org/abs/2306.12891}.
\bibitem[{Blind et~al.(2024)Blind, Gibis, Wenzel, and Beck}]{Blind2024b}
\bibinfo{author}{M.~P. Blind}, \bibinfo{author}{T.~Gibis},
  \bibinfo{author}{C.~Wenzel}, \bibinfo{author}{A.~Beck},
\newblock \bibinfo{title}{{Wall-modeled large eddy simulation of a tandem wing
  configuration in transonic flow}},
\newblock \bibinfo{journal}{Physics of Fluids} \bibinfo{volume}{36}
  (\bibinfo{year}{2024}). \DOIprefix\doi{10.1063/5.0198271}.
\bibitem[{Montoya and Zingg(2022)}]{Montoya2022}
\bibinfo{author}{T.~Montoya}, \bibinfo{author}{D.~W. Zingg},
\newblock \bibinfo{title}{{Stable and Conservative High-Order Methods on
  Triangular Elements Using Tensor-Product Summation-by-Parts Operators}},
\newblock in: \bibinfo{booktitle}{ICCFD11}, \bibinfo{year}{2022}, pp.
  \bibinfo{pages}{1--22}.
\bibitem[{Montoya and Zingg(2024)}]{Montoya2024a}
\bibinfo{author}{T.~Montoya}, \bibinfo{author}{D.~W. Zingg},
\newblock \bibinfo{title}{{Efficient entropy-stable discontinuous
  spectral-element methods using tensor-product summation-by-parts operators on
  triangles and tetrahedra}},
\newblock \bibinfo{journal}{Journal of Computational Physics}
  \bibinfo{volume}{516} (\bibinfo{year}{2024}) \bibinfo{pages}{113360}.
  \DOIprefix\doi{10.1016/j.jcp.2024.113360}.
\bibitem[{Dubiner(1991)}]{Dubiner1991}
\bibinfo{author}{M.~Dubiner},
\newblock \bibinfo{title}{{Spectral methods on triangles and other domains}},
\newblock \bibinfo{journal}{Journal of Scientific Computing}
  \bibinfo{volume}{6} (\bibinfo{year}{1991}) \bibinfo{pages}{345--390}.
  \DOIprefix\doi{10.1007/BF01060030}.
\bibitem[{Lomtev and Karniadakis(1999)}]{Lomtev1999}
\bibinfo{author}{I.~Lomtev}, \bibinfo{author}{G.~E. Karniadakis},
\newblock \bibinfo{title}{{A discontinuous Galerkin method for the
  Navier-Stokes equations}},
\newblock \bibinfo{journal}{International Journal for Numerical Methods in
  Fluids} \bibinfo{volume}{29} (\bibinfo{year}{1999})
  \bibinfo{pages}{587--603}.
  \DOIprefix\doi{10.1002/(SICI)1097-0363(19990315)29:5<587::AID-FLD805>3.0.CO;2-K}.
\bibitem[{Warburton et~al.(1999)Warburton, Lomtev, Du, Sherwin, and
  Karniadakis}]{Warburton1999}
\bibinfo{author}{T.~Warburton}, \bibinfo{author}{I.~Lomtev},
  \bibinfo{author}{Y.~Du}, \bibinfo{author}{S.~Sherwin},
  \bibinfo{author}{G.~Karniadakis},
\newblock \bibinfo{title}{{Galerkin and discontinuous Galerkin spectral/hp
  methods}},
\newblock \bibinfo{journal}{Computer Methods in Applied Mechanics and
  Engineering} \bibinfo{volume}{175} (\bibinfo{year}{1999})
  \bibinfo{pages}{343--359}. \DOIprefix\doi{10.1016/S0045-7825(98)00360-0}.
\bibitem[{Kirby et~al.(2000)Kirby, Warburton, Lomtev, and
  Karniadakis}]{Kirby2000}
\bibinfo{author}{R.~M. Kirby}, \bibinfo{author}{T.~C. Warburton},
  \bibinfo{author}{I.~Lomtev}, \bibinfo{author}{G.~E. Karniadakis},
\newblock \bibinfo{title}{{Discontinuous Galerkin spectral/hp method on hybrid
  grids}},
\newblock \bibinfo{journal}{Applied Numerical Mathematics} \bibinfo{volume}{33}
  (\bibinfo{year}{2000}) \bibinfo{pages}{393--405}.
  \DOIprefix\doi{10.1016/S0168-9274(99)00106-3}.
\bibitem[{Chan et~al.(2016)Chan, Wang, Modave, Remacle, and
  Warburton}]{Chan2016}
\bibinfo{author}{J.~Chan}, \bibinfo{author}{Z.~Wang},
  \bibinfo{author}{A.~Modave}, \bibinfo{author}{J.-F. Remacle},
  \bibinfo{author}{T.~Warburton},
\newblock \bibinfo{title}{{GPU-accelerated discontinuous Galerkin methods on
  hybrid meshes}},
\newblock \bibinfo{journal}{Journal of Computational Physics}
  \bibinfo{volume}{318} (\bibinfo{year}{2016}) \bibinfo{pages}{142--168}.
  \DOIprefix\doi{10.1016/j.jcp.2016.04.003}.
\bibitem[{Duffy(1982)}]{Duffy1982}
\bibinfo{author}{M.~G. Duffy},
\newblock \bibinfo{title}{{Quadrature Over a Pyramid or Cube of Integrands with
  a Singularity at a Vertex}},
\newblock \bibinfo{journal}{SIAM Journal on Numerical Analysis}
  \bibinfo{volume}{19} (\bibinfo{year}{1982}) \bibinfo{pages}{1260--1262}.
  \DOIprefix\doi{10.1137/0719090}.
\bibitem[{Proriol(1957)}]{Proriol1957}
\bibinfo{author}{J.~Proriol},
\newblock \bibinfo{title}{Sur une famille de polynomes {\'a} deux variables
  orthogonaux dans un triangle},
\newblock \bibinfo{journal}{Comptes rendus hebdomadaires des seances de l
  academie des sciences} \bibinfo{volume}{245} (\bibinfo{year}{1957})
  \bibinfo{pages}{2459--2461}.
\bibitem[{Koornwinder(1975)}]{Koornwinder1975}
\bibinfo{author}{T.~Koornwinder},
\newblock \bibinfo{title}{{Two-Variable Analogues of the Classical Orthogonal
  Polynomials}},
\newblock in: \bibinfo{booktitle}{Theory and Application of Special Functions},
  \bibinfo{publisher}{Elsevier}, \bibinfo{year}{1975}, pp.
  \bibinfo{pages}{435--495}.
  \DOIprefix\doi{10.1016/B978-0-12-064850-4.50015-X}.
\bibitem[{Kemm et~al.(2020)Kemm, Gaburro, Thein, and Dumbser}]{Kemm2020}
\bibinfo{author}{F.~Kemm}, \bibinfo{author}{E.~Gaburro},
  \bibinfo{author}{F.~Thein}, \bibinfo{author}{M.~Dumbser},
\newblock \bibinfo{title}{{A simple diffuse interface approach for compressible
  flows around moving solids of arbitrary shape based on a reduced
  Baer–Nunziato model}},
\newblock \bibinfo{journal}{Computers and Fluids} \bibinfo{volume}{204}
  (\bibinfo{year}{2020}). \DOIprefix\doi{10.1016/j.compfluid.2020.104536}.
\bibitem[{Kummer et~al.(2021)Kummer, Weber, and Smuda}]{Kummer2021}
\bibinfo{author}{F.~Kummer}, \bibinfo{author}{J.~Weber},
  \bibinfo{author}{M.~Smuda},
\newblock \bibinfo{title}{{BoSSS: A package for multigrid extended
  discontinuous Galerkin methods}},
\newblock \bibinfo{journal}{Computers and Mathematics with Applications}
  \bibinfo{volume}{81} (\bibinfo{year}{2021}) \bibinfo{pages}{237--257}.
  \DOIprefix\doi{10.1016/j.camwa.2020.05.001}.
\bibitem[{Henneaux et~al.(2023)Henneaux, Schrooyen, Chatelain, and
  Magin}]{Henneaux2023}
\bibinfo{author}{D.~Henneaux}, \bibinfo{author}{P.~Schrooyen},
  \bibinfo{author}{P.~Chatelain}, \bibinfo{author}{T.~Magin},
\newblock \bibinfo{title}{{High-order enforcement of jumps conditions between
  compressible viscous phases: An extended interior penalty discontinuous
  Galerkin method for sharp interface simulation}},
\newblock \bibinfo{journal}{Computer Methods in Applied Mechanics and
  Engineering} \bibinfo{volume}{415} (\bibinfo{year}{2023})
  \bibinfo{pages}{116215}. \DOIprefix\doi{10.1016/j.cma.2023.116215}.
\bibitem[{May and Streitb{\"{u}}rger(2022)}]{May2022}
\bibinfo{author}{S.~May}, \bibinfo{author}{F.~Streitb{\"{u}}rger},
\newblock \bibinfo{title}{{DoD Stabilization for non-linear hyperbolic
  conservation laws on cut cell meshes in one dimension}},
\newblock \bibinfo{journal}{Applied Mathematics and Computation}
  \bibinfo{volume}{419} (\bibinfo{year}{2022}) \bibinfo{pages}{126854}.
  \DOIprefix\doi{10.1016/j.amc.2021.126854}.
\bibitem[{Taylor and Chan(2024)}]{Taylor2024}
\bibinfo{author}{C.~G. Taylor}, \bibinfo{author}{J.~Chan},
\newblock \bibinfo{title}{{An Entropy Stable High-Order Discontinuous Galerkin
  Method on Cut Meshes}}  (\bibinfo{year}{2024}). \URLprefix
  \url{http://arxiv.org/abs/2412.13002}.
\bibitem[{Worku et~al.(2025)Worku, Hicken, and Zingg}]{Worku2025}
\bibinfo{author}{Z.~A. Worku}, \bibinfo{author}{J.~E. Hicken},
  \bibinfo{author}{D.~W. Zingg},
\newblock \bibinfo{title}{{Tensor-product split-simplex summation-by-parts
  operators}},
\newblock \bibinfo{journal}{Journal of Computational Physics}
  \bibinfo{volume}{527} (\bibinfo{year}{2025}) \bibinfo{pages}{113796}.
  \DOIprefix\doi{10.1016/j.jcp.2025.113796}.
\bibitem[{Chan et~al.(2022)Chan, Ranocha, Rueda-Ram{\'{i}}rez, Gassner, and
  Warburton}]{Chan2022a}
\bibinfo{author}{J.~Chan}, \bibinfo{author}{H.~Ranocha}, \bibinfo{author}{A.~M.
  Rueda-Ram{\'{i}}rez}, \bibinfo{author}{G.~Gassner},
  \bibinfo{author}{T.~Warburton},
\newblock \bibinfo{title}{{On the Entropy Projection and the Robustness of High
  Order Entropy Stable Discontinuous Galerkin Schemes for Under-Resolved
  Flows}},
\newblock \bibinfo{journal}{Frontiers in Physics} \bibinfo{volume}{10}
  (\bibinfo{year}{2022}) \bibinfo{pages}{1--31}.
  \DOIprefix\doi{10.3389/fphy.2022.898028}.
\bibitem[{Schwarz et~al.(2025)Schwarz, Kempf, Keim, Kopper, Rohde, and
  Beck}]{Schwarz2025}
\bibinfo{author}{A.~Schwarz}, \bibinfo{author}{D.~Kempf},
  \bibinfo{author}{J.~Keim}, \bibinfo{author}{P.~Kopper},
  \bibinfo{author}{C.~Rohde}, \bibinfo{author}{A.~Beck},
\newblock \bibinfo{title}{{Comparison of Entropy Stable Collocation High-Order
  DG Methods for Compressible Turbulent Flows}}  (\bibinfo{year}{2025}).
  \URLprefix \url{http://arxiv.org/abs/2504.00173}.
\bibitem[{Bassi and Rebay(1997)}]{Bassi1997}
\bibinfo{author}{F.~Bassi}, \bibinfo{author}{S.~Rebay},
\newblock \bibinfo{title}{{A High-Order Accurate Discontinuous Finite Element
  Method for the Numerical Solution of the Compressible Navier-Stokes
  Equations}},
\newblock \bibinfo{journal}{Journal of Computational Physics}
  \bibinfo{volume}{131} (\bibinfo{year}{1997}) \bibinfo{pages}{267--279}.
  \DOIprefix\doi{10.1006/jcph.1996.5572}.
\bibitem[{Gassner et~al.(2009)Gassner, L{\"{o}}rcher, Munz, and
  Hesthaven}]{Gassner2009}
\bibinfo{author}{G.~J. Gassner}, \bibinfo{author}{F.~L{\"{o}}rcher},
  \bibinfo{author}{C.-D. Munz}, \bibinfo{author}{J.~S. Hesthaven},
\newblock \bibinfo{title}{{Polymorphic nodal elements and their application in
  discontinuous Galerkin methods}},
\newblock \bibinfo{journal}{Journal of Computational Physics}
  \bibinfo{volume}{228} (\bibinfo{year}{2009}) \bibinfo{pages}{1573--1590}.
  \DOIprefix\doi{10.1016/j.jcp.2008.11.012}.
\bibitem[{Dalcin et~al.(2019)Dalcin, Rojas, Zampini, {Del Rey Fern{\'{a}}ndez},
  Carpenter, and Parsani}]{Dalcin2019}
\bibinfo{author}{L.~Dalcin}, \bibinfo{author}{D.~Rojas},
  \bibinfo{author}{S.~Zampini}, \bibinfo{author}{D.~C. {Del Rey
  Fern{\'{a}}ndez}}, \bibinfo{author}{M.~H. Carpenter},
  \bibinfo{author}{M.~Parsani},
\newblock \bibinfo{title}{{Conservative and entropy stable solid wall boundary
  conditions for the compressible Navier–Stokes equations: Adiabatic wall and
  heat entropy transfer}},
\newblock \bibinfo{journal}{Journal of Computational Physics}
  \bibinfo{volume}{397} (\bibinfo{year}{2019}) \bibinfo{pages}{108775}.
  \DOIprefix\doi{10.1016/j.jcp.2019.06.051}.
\bibitem[{Chan et~al.(2022)Chan, Lin, and Warburton}]{Chan2022b}
\bibinfo{author}{J.~Chan}, \bibinfo{author}{Y.~Lin},
  \bibinfo{author}{T.~Warburton},
\newblock \bibinfo{title}{{Entropy stable modal discontinuous Galerkin schemes
  and wall boundary conditions for the compressible Navier-Stokes equations}},
\newblock \bibinfo{journal}{Journal of Computational Physics}
  \bibinfo{volume}{448} (\bibinfo{year}{2022}) \bibinfo{pages}{110723}.
  \DOIprefix\doi{10.1016/j.jcp.2021.110723}.
\bibitem[{Hesthaven and Warburton(2008)}]{Hesthaven2008}
\bibinfo{author}{J.~S. Hesthaven}, \bibinfo{author}{T.~Warburton},
  \bibinfo{title}{{Nodal Discontinuous Galerkin Methods}},
  volume~\bibinfo{volume}{54} of \textit{\bibinfo{series}{Texts in Applied
  Mathematics}}, \bibinfo{publisher}{Springer New York}, \bibinfo{address}{New
  York, NY}, \bibinfo{year}{2008}. \DOIprefix\doi{10.1007/978-0-387-72067-8}.
\bibitem[{Karniadakis and Sherwin(2005)}]{Karniadakis2005}
\bibinfo{author}{G.~Karniadakis}, \bibinfo{author}{S.~Sherwin},
  \bibinfo{title}{{Spectral/hp Element Methods for Computational Fluid
  Dynamics}}, \bibinfo{publisher}{Oxford University Press},
  \bibinfo{year}{2005}.
  \DOIprefix\doi{10.1093/acprof:oso/9780198528692.001.0001}.
\bibitem[{Bergot and Durufl{\'{e}}(2013)}]{Bergot2013}
\bibinfo{author}{M.~Bergot}, \bibinfo{author}{M.~Durufl{\'{e}}},
\newblock \bibinfo{title}{{High-order optimal edge elements for pyramids ,
  prisms and hexahedra l}},
\newblock \bibinfo{journal}{Journal of Computational Physics}
  \bibinfo{volume}{232} (\bibinfo{year}{2013}) \bibinfo{pages}{189--213}.
  \DOIprefix\doi{10.1016/j.jcp.2012.08.005}.
\bibitem[{Kopriva(2009)}]{Kopriva2009}
\bibinfo{author}{D.~A. Kopriva}, \bibinfo{title}{Implementing spectral methods
  for partial differential equations: Algorithms for scientists and engineers},
  \bibinfo{publisher}{Springer Science \& Business Media},
  \bibinfo{year}{2009}.
\bibitem[{Kopriva(2006)}]{Kopriva2006}
\bibinfo{author}{D.~A. Kopriva},
\newblock \bibinfo{title}{Metric identities and the discontinuous spectral
  element method on curvilinear meshes},
\newblock \bibinfo{journal}{Journal of Scientific Computing}
  \bibinfo{volume}{26} (\bibinfo{year}{2006}) \bibinfo{pages}{301}.
\bibitem[{Harten and Hyman(1983)}]{Harten1983b}
\bibinfo{author}{A.~Harten}, \bibinfo{author}{J.~M. Hyman},
\newblock \bibinfo{title}{Self adjusting grid methods for one-dimensional
  hyperbolic conservation laws},
\newblock \bibinfo{journal}{Journal of Computational Physics}
  \bibinfo{volume}{50} (\bibinfo{year}{1983}) \bibinfo{pages}{235--269}.
  \DOIprefix\doi{10.1016/0021-9991(83)90066-9}.
\bibitem[{Jameson(2008)}]{Jameson2008}
\bibinfo{author}{A.~Jameson},
\newblock \bibinfo{title}{Formulation of kinetic energy preserving conservative
  sche\-mes for gas dynamics and direct numerical simulation of one-dimensional
  viscous compressible flow in a shock tube using entropy and kinetic energy
  preserving schemes},
\newblock \bibinfo{journal}{J. Sci. Comput.} \bibinfo{volume}{34}
  (\bibinfo{year}{2008}) \bibinfo{pages}{188--208}.
\bibitem[{Hindenlang and Gassner(2020)}]{Hindenlang2019}
\bibinfo{author}{F.~J. Hindenlang}, \bibinfo{author}{G.~J. Gassner},
\newblock \bibinfo{title}{{On the order reduction of entropy stable DGSEM for
  the compressible Euler equations}},
\newblock \bibinfo{journal}{Lecture Notes in Computational Science and
  Engineering} \bibinfo{volume}{134} (\bibinfo{year}{2020})
  \bibinfo{pages}{21--44}. \DOIprefix\doi{10.1007/978-3-030-39647-3_2}.
\bibitem[{Warburton et~al.(1995)Warburton, Sherwin, and
  Karniadakis}]{Warburton1995}
\bibinfo{author}{T.~C. Warburton}, \bibinfo{author}{S.~J. Sherwin},
  \bibinfo{author}{G.~E. Karniadakis},
\newblock \bibinfo{title}{{Unstructured hp/Spectral Elements: Connectivity and
  Optimal Ordering}},
\newblock in: \bibinfo{booktitle}{Computational Mechanics '95},
  \bibinfo{publisher}{Springer Berlin Heidelberg}, \bibinfo{address}{Berlin,
  Heidelberg}, \bibinfo{year}{1995}, pp. \bibinfo{pages}{433--444}.
  \DOIprefix\doi{10.1007/978-3-642-79654-8_72}.
\bibitem[{Xiao and Gimbutas(2010)}]{Xiao2010}
\bibinfo{author}{H.~Xiao}, \bibinfo{author}{Z.~Gimbutas},
\newblock \bibinfo{title}{{A numerical algorithm for the construction of
  efficient quadrature rules in two and higher dimensions}},
\newblock \bibinfo{journal}{Computers {\&} Mathematics with Applications}
  \bibinfo{volume}{59} (\bibinfo{year}{2010}) \bibinfo{pages}{663--676}.
  \DOIprefix\doi{https://doi.org/10.1016/j.camwa.2009.10.027}.
\bibitem[{Krais et~al.(2021)Krais, Beck, Bolemann, Frank, Flad, Gassner,
  Hindenlang, Hoffmann, Kuhn, Sonntag, and Munz}]{Krais2021}
\bibinfo{author}{N.~Krais}, \bibinfo{author}{A.~Beck},
  \bibinfo{author}{T.~Bolemann}, \bibinfo{author}{H.~Frank},
  \bibinfo{author}{D.~Flad}, \bibinfo{author}{G.~Gassner},
  \bibinfo{author}{F.~Hindenlang}, \bibinfo{author}{M.~Hoffmann},
  \bibinfo{author}{T.~Kuhn}, \bibinfo{author}{M.~Sonntag},
  \bibinfo{author}{C.-D. Munz},
\newblock \bibinfo{title}{{FLEXI}: A high order discontinuous {G}alerkin
  framework for hyperbolic{\textendash}parabolic conservation laws},
\newblock \bibinfo{journal}{Comput. Math. with Appl.} \bibinfo{volume}{81}
  (\bibinfo{year}{2021}) \bibinfo{pages}{186--219}.
\bibitem[{Chan and Wilcox(2019)}]{Chan2019a}
\bibinfo{author}{J.~Chan}, \bibinfo{author}{L.~C. Wilcox},
\newblock \bibinfo{title}{{On discretely entropy stable weight-adjusted
  discontinuous Galerkin methods : curvilinear meshes}},
\newblock \bibinfo{journal}{Journal of Computational Physics}
  \bibinfo{volume}{378} (\bibinfo{year}{2019}) \bibinfo{pages}{366--393}.
  \DOIprefix\doi{10.1016/j.jcp.2018.11.010}.
\bibitem[{Chan et~al.(2021)Chan, Bencomo, and {Del Rey
  Fern{\'{a}}ndez}}]{Chan2021}
\bibinfo{author}{J.~Chan}, \bibinfo{author}{M.~J. Bencomo},
  \bibinfo{author}{D.~C. {Del Rey Fern{\'{a}}ndez}},
\newblock \bibinfo{title}{{Mortar-based Entropy-Stable Discontinuous Galerkin
  Methods on Non-conforming Quadrilateral and Hexahedral Meshes}},
\newblock \bibinfo{journal}{Journal of Scientific Computing}
  \bibinfo{volume}{89} (\bibinfo{year}{2021}) \bibinfo{pages}{1--33}.
  \DOIprefix\doi{10.1007/s10915-021-01652-3}.
\bibitem[{Hindenlang et~al.(2012)Hindenlang, Gassner, Altmann, Beck,
  Staudenmaier, and Munz}]{Hindenlang2012}
\bibinfo{author}{F.~Hindenlang}, \bibinfo{author}{G.~J. Gassner},
  \bibinfo{author}{C.~Altmann}, \bibinfo{author}{A.~Beck},
  \bibinfo{author}{M.~Staudenmaier}, \bibinfo{author}{C.-D. Munz},
\newblock \bibinfo{title}{{Explicit discontinuous Galerkin methods for unsteady
  problems}},
\newblock \bibinfo{journal}{Computers {\&} Fluids} \bibinfo{volume}{61}
  (\bibinfo{year}{2012}) \bibinfo{pages}{86--93}.
  \DOIprefix\doi{10.1016/j.compfluid.2012.03.006}.
\bibitem[{Taylor and Green(1937)}]{Taylor1937}
\bibinfo{author}{G.~I. Taylor}, \bibinfo{author}{A.~E. Green},
\newblock \bibinfo{title}{Mechanism of the production of small eddies from
  large ones},
\newblock \bibinfo{journal}{Proceedings of the Royal Society of London. Series
  A-Mathematical and Physical Sciences} \bibinfo{volume}{158}
  (\bibinfo{year}{1937}) \bibinfo{pages}{499--521}.
\bibitem[{Chandrashekar(2013)}]{Chandrashekar2013}
\bibinfo{author}{P.~Chandrashekar},
\newblock \bibinfo{title}{Kinetic energy preserving and entropy stable finite
  volume schemes for compressible {Euler} and {Navier}-{Stokes} equations},
\newblock \bibinfo{journal}{Commun. Comput. Phys.} \bibinfo{volume}{14}
  (\bibinfo{year}{2013}) \bibinfo{pages}{1252--1286}.
\bibitem[{DeBonis(2013)}]{Debonis2013}
\bibinfo{author}{J.~DeBonis},
\newblock \bibinfo{title}{Solutions of the {T}aylor--{G}reen vortex problem
  using high-resolution explicit finite difference methods},
\newblock in: \bibinfo{booktitle}{51st AIAA Aerospace Sciences Meeting},
  \bibinfo{year}{2013}, p. \bibinfo{pages}{382}.
\bibitem[{Vassberg et~al.(2012)Vassberg, Dehaan, Rivers, and
  Wahls}]{Vassberg2008}
\bibinfo{author}{J.~Vassberg}, \bibinfo{author}{M.~Dehaan},
  \bibinfo{author}{M.~Rivers}, \bibinfo{author}{R.~Wahls},
  \bibinfo{title}{Development of a Common Research Model for Applied {CFD}
  Validation Studies}, \bibinfo{year}{2012}.
  \DOIprefix\doi{10.2514/6.2008-6919}.

\end{thebibliography}

\end{document}